\documentclass{amsart}
\usepackage{amsmath,amsfonts,amstext,amsbsy,amssymb,amsthm,flafter}
\usepackage{url}

\usepackage{amsrefs}
\usepackage{xypic}

\newtheorem{theorem}{Theorem}
\renewcommand\thetheorem{\Alph{theorem}}

\theoremstyle{plain}
\newtheorem{lemma}{Lemma}
\newtheorem{corollary}{Corollary}
\newtheorem{proposition}{Proposition}
\newtheorem{axiom}{Axiom}
\theoremstyle{remark}
\newtheorem{remark}{Remark}
\newtheorem*{definition}{Definition}

\newcommand{\C}{\ensuremath{\mathbb{C}}}
\newcommand{\Z}{\ensuremath{\mathbb{Z}}}
\newcommand{\R}{\ensuremath{\mathbb{R}}}
\newcommand{\A}{\ensuremath{\mathcal{A}}}
\newcommand{\T}{\ensuremath{\mathbb{T}}}
\newcommand{\N}{\ensuremath{\mathbb{N}}}
\newcommand{\Hi}{\ensuremath{\mathcal{H}}}
\newcommand{\At}{\A (\T_\theta^n)}
\newcommand{\Atc}{A (\T_\theta^n)}
\newcommand{\Att}{A (\T_\theta^2)}

\newcommand{\vmu}{\boldsymbol{\mu}}
\newcommand{\vnu}{\boldsymbol{\nu}}
\newcommand{\vx}{\mathbf{x}}
\newcommand{\vy}{\mathbf{y}}
\newcommand{\vk}{\mathbf{k}}
\newcommand{\vo}{\mathbf{0}}
\newcommand{\vtau}{\boldsymbol{\tau}}
\newcommand{\vphi}{\boldsymbol{\phi}}
\newcommand{\veps}{\boldsymbol{\epsilon}}
\newcommand{\vm}{\mathbf{m}}
\newcommand{\vp}{\mathbf{p}}
\newcommand{\va}{\mathbf{a}}
\newcommand{\vepst}{\boldsymbol{\tilde{\epsilon}}}
\newcommand{\vdelta}{\boldsymbol{\delta}}
\newcommand{\Id}{\text{Id}}

\newcommand{\cd}{\cdot}

\newcommand{\pmJJ}{\epsilon_J}
\newcommand{\pmJD}{\epsilon_D}
\newcommand{\pmJG}{\epsilon_{\Gamma}}

\newcommand{\half}{\frac{1}{2}}

\newcommand{\am}{\mathbf{A}}
\newcommand{\bm}{\mathbf{B}}

\makeatletter
\def\url@leostyle{%
  \@ifundefined{selectfont}{\def\UrlFont{\sf}}{\def\UrlFont{\small\ttfamily}}}
\makeatother
\def\Xint#1{\mathchoice%
  {\XXint\displaystyle\textstyle{#1}}%
  {\XXint\textstyle\scriptstyle{#1}}%
  {\XXint\scriptstyle\scriptscriptstyle{#1}}%
  {\XXint\scriptscriptstyle\scriptscriptstyle{#1}}%
  \!\int}
\def\XXint#1#2#3{{\setbox0=\hbox{$#1{#2#3}{\int}$}
     \vcenter{\hbox{$#2#3$}}\kern-.5\wd0}}

\def\dashint{\Xint-}

\let\l\left
\let\r\right

\title[Classification of spin structures on the noncommutative \(n\)-torus]{Classification of spin structures on the noncommutative \(n\)-torus}
\author{Jan Jitse Venselaar}
\date{\today}
\address{Mathematical Institute\\
         Utrecht University\\
	 PO Box 80010, 3508 TA Utrecht\\
	 The Netherlands}
\email{J.J.Venselaar1@uu.nl}
\thanks{The author thanks Gunther Cornelissen for discussions and encouragement, and the referee for pertinent remarks}

\begin{document}

\begin{abstract}
 We classify spin structures on the noncommutative torus, and find that the noncommutative \(n\)-torus has \(2^n\) spin structures, corresponding to isospectral deformations of spin structures on the commutative \(n\)-torus. For \(n\geq 4\) the classification depends on Connes' spin manifold theorem. In addition, we study unitary equivalences of these spin structures.
\end{abstract}
\maketitle

\section{Introduction}
The different possible spin structures on a differentiable manifold were classified in the work of Milnor~\cite{MR0157388}; for example, on a (commutative) \(n\)-torus, there exist \(2^n\) inequivalent spin structures. No such general classification of spin structures in currently know in noncommutative geometry --- this amounts to classifying the possible real spectral triple structures on a $C^*$-algebra. 
In this paper we prove that there exist precisely \(2^n\) different real spectral triples on a noncommutative \(n\)-torus, and that these structures are isospectral deformations of spin structures on the commutative \(n\)-torus.
 
The noncommutative torus \(\Atc\), or irrational rotation algebra, is one of the first nontrivial examples of a noncommutative topological manifold, given as a deformation of the usual commutative torus~\cite{MR0227310}~\cite{MR1047281}~\cite{MR781813}. The parameter \(\theta\) is a number for a noncommutative \(2\)-torus, and an antisymmetric \(n\times n\) matrix for higher dimensional tori.

The analog of putting a spin structure and a metric on this algebra is to enhance it into a real spectral triple in the sense of Connes~\cite{connes_noncommutative_1995}~\cite{connes_gravity_1996}. This introduces a set of extra parameters \(\vtau^i\), which are the analogue of the size of the torus. The noncommutative \(n\)-torus, both topologically and with spin structure, has found many applications in mathematical physics, for example \cite{bellissard:5373}, \cite{MR1721895} and \cite{1126-6708-2007-02-033}.
A noncommutative spin structure can certainly be constructed by deforming a spin structure on the commutative \(n\)-torus \cite{MR2551887}, so the question becomes whether this deformation gives all possible spin structures on the noncommutative \(n\)-torus.

In dimension \(2\), the problem was solved by Paschke and Sitarz~\cite{paschke_spin_2006}*{Theorem 2.5}, who showed that a noncommutative \(2\)-torus admits exactly \(4\) different real spectral triples (which are deformations of spin structures on the commutative torus). This result can be reformulated as follows: any real spectral triple which is equivariant with respect to a \(2\)-torus action in the sense of \cite{sitarz_equivariant_2003} (see section~\ref{subsec:equivariant}), is an isospectral deformation of a spin structure on a commutative \(2\)-torus. Note that an equivariant action of \(n\)-torus is different from an \(n\)-torus action as in \cite{MR1937657}, the former is a condition on the spectral triple, the latter is an action along which the algebra is deformed.

Our first result is that the theorem of Paschke and Sitarz holds true in arbitrary dimension:
\begin{theorem}\label{thm:main theorem}
All irreducible real spectral triples with an equivariant \(n\)-torus actions are isospectral deformations of spin structures on an \(n\)-torus.
\end{theorem}
The proof of \cite{paschke_spin_2006}*{Theorem 2.5} does not generalize readily to higher dimensions. Rather than working with the grading operator, which only gives nontrivial conditions in the even-dimensional case, we use the reality operator first. Then we establish that out of several possible candidate structures only one satisfies the growth condition and compact resolvent condition on the Dirac operator. Also, our proof uses at a crucial point Connes' reconstruction theorem~\cite{connes_spectral_2008}*{Theorem 11.5}. We describe the spin structures explicitly in Theorem~\ref{thm:elaborate main theorem}.

In a celebrated paper \cite{MR0139178} (see also \cite{MR0179183}*{Theorem 1}), Adams used the classification of independent vector fields on spheres to deduce elementary results on Radon-Hurwitz numbers of certain classes of matrices. Similarly, our Theorem~\ref{thm:main theorem} can be used to prove the following elementary result on Hermitian matrices, for which we do not know an elementary proof:
\begin{corollary}\label{cor:matrix theorem}
A set of $2^b\times 2^b$ Hermitian matrices ${\{A_i\}}_{i=1}^n$, where $n=2b +1$, such that the equation
\[\det \l(\sum_i x_i A_i\r) = 0,\]
only has the zero solution \({(x_i =0)}_{i=1}^n\) in $\R^n\), generate a Clifford algebra if and only if
\[ \sum_{\sigma \in S_n} \textup{sign}(\sigma) \prod_{i}^n A_{\sigma(i)} = \lambda \Id_k,\]
for some nonzero $\lambda \in \R$.
\end{corollary}

After we have obtained a classification, we study equivalences between different real spectral triples on the same noncommutative \(n\)-torus. For a commutative \(n\)-torus, the diffeomorphisms of a torus act affine on the set of spin structures identified with the vector space \(\mathbb{Z}_2^n\), as shown in \cite{MR860317}. In particular, for the commutative $2$-torus, there are two orbits, one consisting of one element, and the other consisting of three elements. In the case of the noncommutative torus the full diffeomorphism group is not known when \(n>2\). Restricting to inner automorphisms of the algebra, we can show the the following.

\begin{theorem}\label{thm:unitary inequivalence}
Except for a set of \(\theta\) of measure \(0\), the different spin structures of the smooth noncommutative \(n\)-torus \(\At\) cannot be unitarily equivalent by an inner automorphism of the algebra.
\end{theorem}
We also compute the action of unitary transformations induced by some outer automorphisms on the spin structures. These are the action of \(SL(2,\Z)\) on the noncommutative \(2\)-torus and the flip automorphism on the noncommutative \(n\)-torus for \(n > 2\).

The set of \(\theta\) of measure \(0\) in Theorem~\ref{thm:unitary inequivalence} is determined by some Diophantine approximation conditions given in \cite{MR1488068}, and includes the \(\theta\) with only rational entries.

In addition to unitary equivalences of real spectral triples, one could also look at Morita equivalences of real spectral triples
 \cite{vrilly_introduction_2006}*{Chapter 7.2}, even though it is not a true equivalence since it is not symmetric in general \cite{MR2371808}*{Remark 1.143}. It would be interesting to study Morita equivalences of spin structures on noncommutative tori, especially since all Morita equivalences of the algebra of functions on the smooth noncommutative tori are known \cite{rieffel_morita_1999} \cite{elliott_morita_2007}*{Theorem 1.1}. We hope to return to this question in the future.

\section{Definition of a real spectral triple}\label{sec:definition of spectral triple}
Since there are various, slightly different, definitions of real spectral triples, cf. \cite{connes_gravity_1996}*{Pages 159--162},\cite{vrilly_introduction_2006}*{Chapter 3},\cite{bonda_elements_2001}*{Chapter 10}, and we need to refer to the axioms unambiguously, we will explicitly state the definition of real spectral triple we use. For the definition of an equivariant spectral triple, see section~\ref{subsec:equivariant}. A spin structure on a manifold \(M\), of dimension \(n\), is a nontrivial double cover of the principal \(SO(n)\) bundle of orthogonal frames of the tangent bundle \(TM\) \cite{lawson_spin_1989}*{Chapter II.1}. After \cite{connes_spectral_2008}, we know that a real spectral triple is the right noncommutative geometry analog of a spin structure on a manifold.

For a real spectral triple we have the following data:
\begin{itemize}
 \item A unital, Fr\'echet algebra \(\A\), $*$-isomorphic to a proper dense subalgebra of a $C^*$-algebra which is stable with respect to the holomorphic functional calculus,
 \item A separable Hilbert space \(\Hi\) with a representation of \(\A\) acting as bounded operators,
 \item An unbounded self-adjoint operator \(D\), called `Dirac operator',
 \item An antilinear isometry \(J\) of \(\Hi\) onto itself, called `reality operator',
 \item An integer \(n\geq 0\), called `dimension',
 \item If $n$ is even, a self-adjoint unitary operator \(\Gamma\) of \(\Hi\) onto itself, such that \(\Gamma^2 = \Id\), called `grading operator'. We call the the spectral triple \emph{even} in this case.
\end{itemize}
These objects also satisfy the Axioms~\ref{ax:D compact resolvent} to~\ref{ax:poincare} in order for them to be called a real spectral triple of dimension \(n\).

\begin{axiom}[Compact resolvent]\label{ax:D compact resolvent}
 The Dirac operator \(D\) has compact resolvent, that is, \(D\) has finite dimensional kernel and \(D^{-1}\) (defined on the orthogonal complement of the kernel) is a compact operator. Furthermore, for all \(a\in \A\), \([D,\pi(a)]\) is a bounded operator.
\end{axiom}

\begin{axiom}[Grading operator]\label{ax:grading operator}
 If \(n\) is even, the \(\Z_2\) grading operator \(\Gamma\) splits the Hilbert space \(\Hi\) as \(\Hi^+\) and \(\Hi^-\), where \(\Hi^{\pm}\) is the \((\pm)\) eigenspace of \(\Gamma\). The operator \(D\) is odd with respect to this operator, $D\Gamma = - \Gamma D$, and the representation \(\pi\) of \(\A\) on \(\Hi\) is even, so we can write 
\[\pi(a) = \begin{pmatrix} a & 0\\ 0 & a\end{pmatrix} \qquad D = \begin{pmatrix} 0 & D^-\\ D^+ & 0\end{pmatrix},\]
where \(D^+\) and \(D^-\) are adjoint to each other.
\end{axiom}

\begin{axiom}\label{ax:signs of triple}
 The operators \(J\), \(D\) and \(\Gamma\) satisfy the commutation relations from Table~\ref{tab:signs of triple}, and the operator \(J\) is unitary: \(J^\dagger = J^{-1}\).
\end{axiom}

\begin{table}\caption{Signs of the spectral triple}
\begin{tabular}{|cclr|cccccccc|}\hline
 \multicolumn{4}{|c|}{\(n\) mod 8}          & 0 & 1 & 2 & 3 & 4 & 5 & 6 & 7\\\hline
 \(J^2\) & \(=\) & \(\pm \Id\) & (\(\pmJJ\))       & $+$ & $+$ & $-$ & $-$ & $-$ & $-$ & $+$ & $+$\\
 \(JD\) & \(=\) & \(\pm DJ\) & (\(\pmJD\))       & $+$ & $-$ & $+$ & $+$ & $+$ & $-$ & $+$ & $+$\\
 \(J\Gamma\) & \(=\) & \(\pm \Gamma J\) & (\(\pmJG\)) & $+$ &  & $-$ &  & $+$ & & $-$ & \\\hline
\end{tabular}
\label{tab:signs of triple}
\end{table}

\begin{axiom}[Dimension]\label{ax:dimension}
 The eigenvalues \(\mu_k\) of \(|D|^{-1}\), arranged in decreasing order, grow asymptotically as
\[\mu_k = \mathcal{O}(k^{-n}),\] for an integer \(n\) (called the dimension).
\end{axiom}

\begin{axiom}[First order condition]\label{ax:first order condition}
For all \(a,b \in \A\) the following commutation relations hold:
\begin{align}
  [a, J b^* J^\dagger] &= 0.\label{eqn:J maps to commutant}\\
  [ [D,a], J b^* J^\dagger] &= 0.\label{eqn:first order condition}
\end{align}
\end{axiom}
We will write \(b^o = J b^* J^\dagger\). The above formulas establish that the \emph{opposite algebra}:
\[\A^o = \{a^o = J^\dagger a^* J | a\in \A\},\] lies within the commutator of \(\A\).

Recall that a Hochschild \(k\)-chain is defined as an element \(\mathbf{c}\) of \(C_k(M,\A)= M \otimes \A^{\otimes k}\), with \(M\) a bimodule over \(\A\). A boundary map \(b: C_k \rightarrow C_{k-1}\) is defined as 
\begin{align*}
 b &= \sum_{i=0,k} {(-1)}^i d_i,\\
 d_0(m\otimes a_1 \otimes \cdots \otimes a_k) &= m a_1\otimes a_2\cdots \otimes a_k,\\
 d_i (m\otimes a_1 \otimes \cdots\otimes a_k) &= m\otimes a_1 \otimes \cdots \otimes a_i a_{i+1} \otimes \cdots \otimes a_k,\\
 d_k(m\otimes a_1 \otimes \cdots \otimes a_k) &= a_k m \otimes a_1 \otimes \cdots \otimes a_{k-1}.
\end{align*}
Since \(b^2=0\), this makes \(C_k(M,\A)\) into a chain complex.

Axiom~\ref{ax:first order condition} gives a representation of Hochschild \(k\)-chains \(C_k(\A, \A\otimes \A^o)\) on \(\Hi\) by 
\begin{equation}
 \pi_D( (a\otimes b^o) \otimes a_1 \otimes \cdots \otimes a_k) = a b^o [D,a_1]\ldots [D,a_k].\label{eqn:hochschild representation}
\end{equation}

\begin{axiom}[Orientability]\label{ax:hochschild cycle}
 There is a Hochschild cycle \(\mathbf{c} \in Z_n (\A,\A\otimes \A^o)\) such that \(\pi_D (\mathbf{c}) = \Gamma\) when $n$ is even, and \(\pi_D (\mathbf{c}) = \Id\) when $n$ is odd.
\end{axiom}

\begin{axiom}[Regularity]
 For all \(a\in \A\), both \(a\) and \([D,a]\) belong to the domain of smoothness \(\bigcap_{k=1}^\infty \text{Dom}\l(\delta^k\r)\), where the derivation \(\delta\) is given by \(\delta(T) = [|D|,T]\), with \(|D| = \sqrt{D^* D}\).
\label{ax:regularity}
\end{axiom}

\begin{axiom}[Finiteness]
 The space of smooth vectors \(\Hi^\infty = \bigcap_{k=1}^{\infty} \text{Dom}\l(D^k\r)\) is a finitely generated projective left \(\A\) module. Also, there is a Hermitian pairing \((\eta | \xi)\) on this module, given by 
\[\langle \eta, \xi \rangle = \dashint (\eta | \xi) |D|^{-n},\]
where \(\dashint\) is the noncommutative integral (defined for example in \cite{vrilly_introduction_2006}*{Chapter 5}).
\label{ax:finiteness}
\end{axiom}

\begin{axiom}[Poincar\'e duality]
The Fredholm index of the operator \(D\) yields a nondegenerate intersection form on the \(K\)-theory ring of the algebra \(\A\otimes \A^o\).
\label{ax:poincare}
\end{axiom}
Finally, we restrict our attention to \emph{irreducible} spectral triples, that is, the only operators commuting with the action of the algebra and \(D\) are the scalars. In the case \(\A\) is commutative, this is equivalent to demanding the manifold is connected, see  \cite{connes_gravity_1996}*{Remark 6 on p.163}.

\subsection{Equivariant spectral triples}\label{subsec:equivariant}
There are different candidates for the notion of symmetries of noncommutative geometries. One obvious candidate is the group of automorphisms of the algebra $\A$, but for noncommutative algebras, this group can be very small, while it seems that there should be more symmetries available. An attempt to enlarge this group of symmetries in an interesting way is the notion of equivariant spectral triples. Equivariant spectral triples were introduced in \cite{sitarz_equivariant_2003}. One describes symmetries of spectral triples in the form of Hopf algebras. In this paper we are interested in $n$-dimensional spectral triples, which are equivariant with respect to a Hopf algebra with $n$ different commuting derivations. This, together with the irreducibility condition, is the analog of a connected compact homogeneous space in commutative geometry. 

In the context of Hopf algebras, we shall use Sweedler's notation for the coproduct: $\Delta h = h_{(1)} \otimes h_{(2)}$. See for example \cite{MR1321145}*{Chapter 3} for an introduction to Hopf algebras, and some standard notation.
An \emph{equivariant real spectral triple} is a real spectral triple $(\A,\Hi,D,J)$ together with a Hopf algebra $H$, with multiplication $\mu$, unit $\eta$, comultiplication \(\Delta\), counit $\epsilon$ and antipode \(S\), and an antilinear involution ${}^*$ making $H$ into a $*$-algebra such that 
\[ \Delta h^* = {\l(\Delta h\r)}^{*\otimes*} \quad \epsilon(h*) = \overline{ \epsilon(h)} \quad {\l(S\circ *\r)}^2 = \Id.\]
Recall that an $H$-module algebra is an algebra $A$ with a complex linear representation $\rho$ of $H$ on $A$ such that $A$ is a linear space, and $\rho$ respects the algebra structure:
\[
 \rho(h) (a_1 a_2) = \l(\rho(h_{(1)}) a_1\r)\l(\rho(h_{(2)}) a_2\r),
\]
for all $h\in H, a_1, a_2\in A$. When $A$ is an $H$-module algebra we define an equivariant (left,right) $A$-module to be a (left,right) $A$-module $M$ such that 
\[ \rho_M(h) (am) = \l(\rho_A(h_{(1)}) a\r) \l(\rho_M(h_{(2)}) m\r),\]
for all $h\in H, a\in A,m\in M$.
The objects of the equivariant real spectral triple transform in a compatible way under the action of the Hopf algebra:
\begin{itemize}
 \item The algebra $\A$ is an $H$-module algebra.
 \item There is a dense subspace $V\subset \Hi$ such that $V$ is an equivariant left $\A$-module.
 \item For every $h\in H$, the Dirac operator is equivariant, $[D,h] = 0$ on the (dense) intersection of the domain of $D$ and $V$.
 \item The action of $H^{op}$ is well defined on the opposite algebra $\A^{o}$ via the equality $J^{-1} h J = {(Sh)}^*$.
 \item If the spectral triple is even, $[\Gamma,h] = 0$.
\end{itemize}

In the case of equivariance with respect to a torus action, we take the universal enveloping algebra $U(\mathfrak{t}^n)$ of the familiar Lie algebra $\mathfrak{t}^n$ of the $n$-torus. This means that we have a basis of derivations $\delta_i$, with a representation $\rho$ on $\Hi$ such that for $a \in \At$, $v \in \Hi$:%
\begin{subequations}
\begin{align}
 \delta_i \delta_j &= \delta_j \delta_i\label{eqn:commutative hopf algebra},\\
 \Delta(\delta_i) &= \delta_i \otimes \Id + \Id \otimes \delta_i\label{eqn:comultiplication},\\
 \rho(\delta_i)\pi(a) v &= \l(\pi( \delta_i a ) + \pi(a) \r) v\label{eqn:equivariance algebra},\\
 \rho(\delta_i) D v &=  D \rho(\delta_i) v\label{eqn:equivariance dirac operator},\\ 
 \rho(\delta_i) J v &= - J {\rho(\delta_i)}^* v\label{eqn:equivariance reality operator}.
\end{align}
\end{subequations}

\section{Outline of the classification}
The outline of the proof of Theorem~\ref{thm:main theorem} is as follows. First in section~\ref{sec:hilbert space and algebra} we determine the action of the algebra on the Hilbert space, such that the equivariance condition is met. This action is already well-known, but we derive it to show that there are no other possibilities. In section~\ref{sec:reality operator} we move to real spectral triples, and determine possible forms of the reality operator $J$, by considering the anti-isomorphism $\At \mapsto \At^o$ and the equivariance condition~\eqref{eqn:equivariance reality operator}. We find several possible families of real spectral triples, only one of which consists of isospectral deformations of spin structures on the commutative torus. In the next section, section~\ref{sec:dirac operator}, we determine the classes of possible Dirac operators for each candidate family of real spectral triples using equivariance of the Dirac operator and the first-order condition, and show that only the isospectral deformation family is compatible with the compact resolvent condition.

In section~\ref{sec:hochschild homology} we determine that the parameters $\vtau$ in the Dirac operators must be linearly independent vectors spanning $\R^n$, using the Hochschild homology condition and earlier results on the Hochschild homology on noncommutative tori. The last step in the classification is done in section~\ref{sec:dimension,finiteness,regularity}, where we use the spin manifold reconstruction theorem to show that the Dirac operator is really a Dirac operator in the sense of spin geometry.

After the classification, we give in section~\ref{sec:uniqueness and freedom} an explicit description of the constructed real spectral triples on the noncommutative torus.

Finally, in section~\ref{sec:unitary equivalences}, we study unitary equivalences of the real spectral triples, and show that unitary equivalences induced by inner automorphisms of the algebra do not change the spin structure for almost all $\theta$. When $n=2$ we show that the known outer automorphisms do change the spin structure, if the real spectral triple is an isospectral deformation of a nontrivial spin structure. When $n>2$, we show that the flip automorphism cannot change the spin structure.

\section{Hilbert space and algebra}\label{sec:hilbert space and algebra}
We look for possible equivariant representations of the algebra of functions of the smooth noncommutative torus, and give a basis of the Hilbert space $\Hi$ for which equivariance is obvious. We do not use any special conditions from the definition of a real spectral triple, except that the Hilbert space should be separable. 

We denote the noncommutative torus, or more precisely, the algebra of continuous functions on the noncommutative $n$-torus, by $\Atc$, where $\theta$ is an antisymmetric real $n\times n$ matrix.
As Hopf algebra symmetry for which the algebra representation must be equivariant we take the Hopf algebra generated by $n$ independent commuting elements $\delta_1,\ldots,\delta_n$.

The algebra of smooth or continuous functions on the noncommutative torus is generated by unitary elements $U_{e_1},\ldots, U_{e_n}$ such that 
\[ U_{e_k} U_{e_l} = \exp\l(2\pi i \theta_{kl}\r) U_{e_l} U_{e_k},\]
with $\theta_{kl}$ the component of the matrix at position $(k,l)$. As a short-hand notation, we will write \[e( \cdot) = \exp(2 \pi i \cdot ).\]
The Hopf algebra action of our basis elements on the unitary generators is \cite{sitarz_equivariant_2003}:
$\delta_i U_{e_j} = U_{e_j}$ if $i=j$, and $0$ otherwise.
If we interpret the $e_j$ as the $j$-th basis vector of $\Z^n$, we can write more generally unitary elements of the algebra as:
\[U_{\vx} = e\l( \half \sum_{k>j} x_j \theta_{jk} x_k\r) {\l(U_{e_1}\r)}^{x_1} {\l(U_{e_2}\r)}^{x_2} \cdots {\l(U_{e_n}\r)}^{x_n},\]
for $\vx = \sum x_j e_j$.
The Fr\'echet algebra of smooth functions on the noncommutative torus is a dense subalgebra of this algebra.

Just as for the algebra, we will write $\delta_{\vx}$ for the derivation given by 
\[\delta_{\vx} = \delta_{x_1} \delta_{x_2} \cdots \delta_{x_n}.\]
From the definitions, it is immediate that
\begin{equation}
 \delta_{\vx} U_{\vy} = \l(\vx\cdot\vy\r) U_{\vy},
\end{equation}
where $\l(\vx\cdot\vy\r)$ is the standard inner product on $\Z^n$.

Now we look for a Hilbert space $\Hi_0$ which is an equivariant left $A_\theta$-module. As a basis of $\Hi_0$ we choose mutual eigenvectors $e_{\vmu}$ of the derivations:
\[\rho(\delta_i) e_{\vmu} = \mu_i e_{\vmu}.\]
where the $\vmu$ form a countable subset of $\R^n$.

In order for the spectral triple to be a noncommutative torus, we demand that the action of the algebra is equivariant with respect to a torus action, as in equation~\eqref{eqn:equivariance algebra}.
Written out for the unitary generators $U_{\vx}$, we see:
\[\pi_0(U_{\vx}) e_{\vmu} = u_{\vx,\vmu} e_{\vmu+\vx},\]
with $u_{\vx,\vmu} \in \C$ to be determined.
Thus for the minimal irreducible equivariant representation the $\vmu$ will lie in a translate of a lattice:
\begin{equation}\Hi_0 = \bigoplus_{\vm\in \Z^n} \Hi_{\vmu_0 + \vm},\label{eqn:irred hilbert space}\end{equation}
where each $\Hi_{\vmu_0+\vm} \cong \C$. There are no restrictions yet on $\vmu_0$, these will be determined later.

Since the $U_{\vx}$ should be unitary, we have that 
\begin{align}\langle e_{\vnu}, u_{\vx,\vmu} e_{\vmu+\vx}\rangle &= \langle u_{-\vx,\vnu} e_{\vnu-\vx}, e_{\vmu}\rangle \notag\\
\Rightarrow u_{\vx,\vmu} \delta_{\vnu,\vmu+\vx} &= \overline{u_{-\vx,\vnu}} \delta_{\vnu-\vx,\vmu}.\label{eqn:unitarity relations}\end{align}
Finally the definition of $U_{\vx}$ in terms of $U_i$ gives the relations
\begin{equation}
 u_{\vx+\vy,\vmu} = e(\half \vx \cdot \theta \vy) u_{\vx,\vmu+\vy} u_{\vy,\vmu}.\label{eqn:linearity relations}
\end{equation}

\begin{lemma}\label{lem:algebra representation}
Up to unitary transformations of $\Hi$ any unitary equivariant representation of $\At$ on $\Hi$ is given by
\begin{equation}
\pi_0^\am(U_{\vx}) e_{\vmu}= e\l(\half\vx \cd \am \vx + \vx \cd \am \vmu \r) e_{\vmu+\vx},\label{eqn:short form for generators}
\end{equation}
with $\am$ any $n\times n$ matrix such that $\am - \am^t = \theta$. 
\end{lemma}
Since the representations $\pi_0^{\am}$, given by different matrices $\am$ such that $\am - \am^t = \theta$ are equivalent, we drop the $\am$ from the notation, and just write $\pi_0$ for a representation defined by equation~\eqref{eqn:short form for generators}.
\begin{proof}
It is clear that for any matrix $\am$ the representation given above satisfies the relations~\eqref{eqn:unitarity relations} and \eqref{eqn:linearity relations}.
Given two representations $\pi$ and $\pi'$ satisfying the equivariance condition~\eqref{eqn:equivariance algebra} we can write any element $w\in \Hi$ as a unique sum $\sum_{\vx\in \Z^n} \lambda_{\vx} \pi(U_{\vx}) e_0$ with ${(\lambda)}_{\vx\in\Z^n} \in \ell^2(\Z^n)$. In other words, $e_0$ is a cyclic vector with respect to the algebra action. Now construct an operator $T:\Hi \rightarrow \Hi$ by setting $T e_0 = e_0$ and extending by \[T w = T \sum_{\vx\in \Z^n} \lambda_{\vx} \pi(U_{\vx}) e_0 = \sum_{\vx\in \Z^n} \lambda_{\vx} \pi'(U_{\vx}) e_0.\] 
This is well defined if $\pi$ and $\pi'$ are representations of the same algebra $\At$, since they satisfy the same algebra relation~\eqref{eqn:linearity relations}, and it is an invertible map on $\Hi$, because both $\pi(U_{\vx}) e_0$ and $\pi'(U_{\vx})e_0$ span the Hilbert space if we take all $\vx\in \Z^n$. By construction $T^{-1} \pi'(U_{\vx}) T = \pi(U_{\vx})$, and it is a unitary transformation, because we can calculate:
\begin{align*}
 \langle T v, T w\rangle &= \sum_{\vx,\vy \in \Z^n} \bar{\lambda}_{\vx} \mu_{\vy} \langle \pi'(U_{\vx}) e_0,\pi'(U_{\vy}) e_0\rangle\\
 &= \sum_{\vx - \vy = 0} \bar{\lambda}_{\vx} \mu_{\vy} \langle \pi'(U_{\vx}) e_0,\pi'(U_{\vy})e_0\rangle\\
 &= \sum_{\vx\in \Z^n}\bar{\lambda}_{\vx} \mu_{-\vx}\\
 &= \langle v,w\rangle,
\end{align*}
since $\langle e_{\vx} ,e_{\vy}\rangle = 0$ if $\vx\neq \vy$, and the representations $\pi$ and $\pi'$ are unitary. 
\end{proof}
There might be more unitary equivalences of the algebra, a question which we explore in section~\ref{sec:unitary equivalences}, but for now we were only interested in possible representations of the algebra.

\section{Reality operator}\label{sec:reality operator}
In this section, we derive conditions on the equivariant reality operator $J$, to give us a real spectral triple. We will only consider conditions following from relations between $\A$, $\Hi$ and $J$ and the equivariance condition~\eqref{eqn:equivariance reality operator}. No use is made of the Dirac operator in this section.

The equivariance condition for $J$ is given in~\eqref{eqn:equivariance reality operator}:
\begin{equation} \rho(l) J v = - J \rho(l^*) v,\label{eqn:equivariance J operator}\end{equation}
for $v\in \Hi$, and $l$ an element of the Hopf algebra of the symmetries of the noncommutative $n$-torus. 
For our basic derivations $\delta_i$, we have $\delta_i^* = \delta_i$, so this just means
\begin{equation} \delta_i J e_{\vmu} = - \mu_i J e_{\vmu},\label{eqn:J maps to negative}\end{equation}
for all $i$. Since we are working with the representation given in \eqref{eqn:irred hilbert space}, we see that $J$ must map an element $e_{\vmu}$ of the basis to $e_{-\vmu}$.

The Tomita involution $J_0(a) = a^*$ \cite{MR0270168} gives a commuting representation, but it does not follow the commutation relations in Table~\ref{tab:signs of triple}, so we will have to enlarge our Hilbert space. 
Define 
\begin{equation}\Hi := \bigoplus_{j \in I} \Hi_j,\label{eqn:direct sum}\end{equation}
with each $\Hi_j$ as in \eqref{eqn:irred hilbert space} and $I$ an index set. Every nondegenerate representation of an involutive Banach algebra is a direct sum of cyclic representations \cite{MR548728}*{Proposition I.9.17}, and because of the equivariance condition~\eqref{eqn:equivariance algebra}, we get that the only cyclic representations we can consider are the ones given by Lemma~\ref{lem:algebra representation}. We write basis vectors of this Hilbert space as $e_{\vmu,j}$ with $\vmu \in \Z^n$ and $j\in I$. 

For different $j\in I$, lattices spanned by $\vmu$ could a priori be shifted by a different amount, satisfying \eqref{eqn:J maps to negative}, but we will see in section~\ref{sec:dirac operator} that this cannot be the case if the spectral triple is irreducible.

We look for an antilinear operator $J$ such that $J^2 = \pmJJ \Id$, with signs as in Table~\ref{tab:signs of triple}, and so that for every $a\in \At$, $Ja^* J^{-1}$ commutes with all $b\in \At$, satisfying equation~\eqref{eqn:J maps to commutant}. The image of $\At$ under the isomorphism $a \mapsto J a^* J^{-1}$ is denoted by $\At^o$. We write $U_{\vx}^o$ for the image of $U_{\vx}$.

Firstly, equation~\eqref{eqn:equivariance J operator} has as a consequence that $J$ acts as follows on elements of the basis $e_{\vmu,j}$:
\begin{equation} J e_{\vmu,j} = \sum a_{kj}(\vmu) e_{-\vmu,k}.\label{eqn:action of J on basis}\end{equation}
We can thus write $J e_{\vmu,j} = \Lambda(-\vmu) J_0 e_{\vmu,j}$, with $\Lambda(-\vmu)$ some unitary or skew unitary bounded linear functional, and $J_0$ an antilinear diagonal operator:
\[ J_0 a e_{\vmu,j} = a^* e_{-\vmu,j}.\]

If we apply $J$ twice, we should get $\pmJJ \Id$, which can be written as
\[ J\cdot J e_{\vmu,j} = J \Lambda(-\vmu) e_{-\vmu,j} = \Lambda(\vmu) \Lambda^*(-\vmu) e_{-\vmu,j},\]
and so we see that 
\begin{equation}
 \Lambda(\vmu) {\Lambda(-\vmu)}^* = \pmJJ \Id.\label{eqn:j squared}
\end{equation}

By applying $J$ on a unitary generator $U_{\vx}$ of $\A_{\theta}$ we get the following condition on $\Lambda(\vmu)$:
\begin{lemma}\label{lem:commutant property}
 The map $a \mapsto J a^* J^{\dagger}$ is an isomorphism into the commutant, if and only if for all $\vx,\vy \in \Z^n$:%
\begin{subequations}
\begin{equation}\Lambda(\vx + \vy) = e(\vx \cd \am \vy + \vy \cd \am \vx) \Lambda(\vx) {\Lambda(0)}^{\dagger} \Lambda(\vy),\end{equation}
where $J e_{\vmu,j} = \Lambda(-\vmu) J_0 e_{\vmu,j}$, with $J_0$ the Tomita involution, and
\begin{align}\Lambda(\vx) {\Lambda(\vx)}^{\dagger} &= \Id,\\
\Lambda(\vx) {\Lambda(-\vx)}^* &= \pmJJ \Id.\end{align}
\end{subequations}
As a consequence
\begin{equation}
 U_{\vx}^o e_{\vmu,j} = e(\vmu \cd \am \vx + \half\vx \cd \am \vx) \Lambda'(\vx) {\Lambda(0)}^{\dagger} e_{\vmu + \vx,j},\label{eqn:opposite algebra action}
\end{equation}
where $\Lambda(\vmu) = e(\vmu \cd \am \vmu) \Lambda'(\vmu)$, and ${\Lambda'(\vmu)}_{ij} = c_{ij} e(\vphi_{ij}(\vmu))$ with $c_{ij}\in\C$ and $\vphi_{ij}:\R^n\rightarrow \R$ such that $\sum_{j\in I} c_{ij}e(\vphi_{ij}(-\vmu)) c_{jk}^*e(-\vphi_{ij}(\vmu)) =\pmJJ \delta_{ik}$.
\end{lemma}

\begin{proof}
 First we calculate $U_{\vx}^o = J U_{\vx}^* J^{\dagger}$ using equation~\eqref{eqn:action of J on basis}:
\begin{align*}
 U_{\vx}^o e_{\vmu,j} &= J U_{\vx}^* J^{\dagger} e_{\vmu,j}\\
&= J U_{\vx}^* {\Lambda(\vmu)}^t e_{-\vmu,j}\\
&= J e\l(\vx\cd \am\vmu + \half \vx\cd \am \vx\r) {\Lambda(\vmu)}^t e_{-\vmu-\vx,j}\\
&= \Lambda(\vmu+\vx) e\l(-\vx\cd \am \vmu - \half \vx \cd \am \vx \r) {\Lambda(\vmu)}^\dagger e_{\vmu+\vx,j}.\\
\end{align*}
We compute the commutator $[U_{\vy},U_{\vx}^o]$:

\begin{align*}
U_{\vy} U_{\vx}^o e_{\vmu,j} &= e\l(\vy \cd \am (\vmu+\vx + \half \vy) -\vx\cd \am (\vmu +\half\vx)\r) \Lambda(\vmu+\vx) \Lambda^{\dagger}(\vmu) e_{\vmu+\vx+\vy,j}\\
U_{\vx}^o U_{\vy} e_{\vmu,j} &= e\l(-\vx \cd \am (\vmu+\vy + \half \vx)+\vy\cd \am (\vmu + \half \vy)\r) \Lambda(\vmu+\vx+\vy) \Lambda^{\dagger}(\vmu+\vy) e_{\vmu+\vx+\vy,j}.
\end{align*}
If the commutator vanishes, we see that by canceling common terms we must have:
\[ e\l(\vy \cd \am \vx\r) \Lambda(\vmu+\vx) \Lambda^{\dagger}(\vmu) = e\l(-\vx \cd \am \vy\r) \Lambda(\vmu+\vx + \vy) \Lambda^{\dagger}(\vmu+\vy).\]

This has as a consequence that ${\Lambda(\vx)}_{ij}$ consists of $e(f_{ij}(x))$ with $f_{ij}(x)$ a function of the form $\vx\cd \bm_{ij} \vx + \vphi_{ij}(\vx) + \nu_{ij}$. We see that the quadratic part must be the same for each component, and that $\bm_{ij} = \am$. The constant part $\nu_{ij}$ can be absorbed into by unitary transformation. We can thus write 
\begin{equation}\Lambda(\vx) = e(\vx \cd \am \vx) \Lambda'(\vx),\label{eqn:lambda accent}\end{equation}
where $\Lambda'(\vx)$ consists of functions $c_{ij} e(\vphi_{ij}(\vx))$ such that $\Lambda'(\vx) = \pmJJ {\Lambda'(-x)}^t$ and $\Lambda'(\vx) {\Lambda'(\vx)}^{\dagger} = \Id$.
If we then calculate $U_{\vx}^o$:
\begin{align*}
 U_{\vx}^o e_{\vmu} &= \Lambda(\vmu+\vx) e\l(-\vx\cd \am \vmu - \half \vx \cd \am \vx \r) {\Lambda(\vmu)}^\dagger e_{\vmu+\vx,j}\\
&= e\l( (\vmu + \vx)\cd \am(\vmu + \vx) -\vx\cd \am \vmu - \half \vx \cd \am \vx - \vmu \cd \am \vmu\r) \Lambda'(\vmu+\vx) {\Lambda'(\vmu)}^\dagger e_{\vmu+\vx,j}\\
&= e\l(\vmu\cd \am \vx + \half \vx \cd \am \vx \r)\Lambda'(\vmu+\vx) {\Lambda'(\vmu)}^\dagger e_{\vmu+\vx,j}\\
&= e\l(\vmu\cd \am \vx + \half \vx \cd \am \vx \r)\Lambda'(\vx){\Lambda'(0)}^\dagger e_{\vmu+\vx,j}.
\end{align*}
By definition~\eqref{eqn:lambda accent} of $\Lambda'$, we have $\Lambda'(0) = \Lambda(0)$, so this is exactly equation~\eqref{eqn:opposite algebra action}.
\end{proof}

\section{Dirac operator}\label{sec:dirac operator}
The remaining piece of the real spectral triple is the Dirac operator. In this section, using the results from the previous section, the equivariance condition~\eqref{eqn:equivariance dirac operator}, and Axioms~\ref{ax:D compact resolvent} and \ref{ax:first order condition}, we derive conditions for the Dirac operator $D$. We see that for different forms of $\Lambda$ as given in Lemma~\ref{lem:commutant property}, only the form ${\Lambda'(\vx)}^2 =\Id$ can lead to isospectral deformations of spin structures on the commutative $n$-torus. By applying the compact resolvent condition of Axiom~\ref{ax:D compact resolvent}, we show that this is the only possibility compatible with the definition of noncommutative geometry.

An equivariant Dirac operator $D$ commutes with the basic derivations $\delta_i$ as described in \eqref{eqn:equivariance dirac operator}:
\[
 D e_{\vmu,j} = \sum_{k\in I} d_{\vmu,j}^k e_{\vmu,k}.
\]
Since $D$ should be self-adjoint, we have that $d_{\vmu,j}^k = {\l(d_{\vmu,k}^j\r)}^*$. We will write
\begin{equation} D e_{\vmu,j} = D(\vmu) e_{\vmu,j}.\label{eqn:definition D(mu)}\end{equation}
This means that \(D(\vmu)\) is the operator $D$ restricted to the eigenspace of derivations with eigenvalue $\vmu$. This is well defined due to the equivariance condition~\eqref{eqn:equivariance dirac operator}.

Given an element $\vmu$ in a shifted lattice $\tilde{\Z}^n$, we define the Hilbert space $\Hi_{\vmu}$ as the span of eigenvectors $v_j$ of the basic derivations $\delta_i$ such that $\delta_i v_j = \vmu_i v_j$ for each $i$.
As a consequence of the equivariance of the Dirac operator and the irreducibility condition, all the $\vmu$ must lie in the same lattice, by the following argument.

Between two Hilbert spaces $\Hi_{\vmu}$ and $\Hi_{\vnu}$ we have an isometry if $\vmu - \vnu \in \Z^n$, given by the unitary element $U_{\vmu-\vnu}$ ofthe algebra. Now consider the projector $P_{\vmu}$ that is $\text{Id}$ on $\Hi_{\vnu}$ for which $\vmu-\vnu \in\Z^n$, and $0$ otherwise. This projection clearly commutes with the algebra, and, because of \eqref{eqn:definition D(mu)}, also with the Dirac operator. If the spectral triple is irreducible, only scalars may commute with both the algebra and the Dirac operator, so $P_{\vmu}$ is the identity on the whole Hilbert space $\Hi$, hence all lattices are shifted by the same vector $\veps$. Because of \eqref{eqn:J maps to negative}, we can then conclude that $\veps$ consists of elements which are either $0$ or $\half$.

From the first order condition in equation~\eqref{eqn:first order condition}, we deduce:
\begin{lemma}\label{lem:first order Dirac operator}
 An equivariant Dirac operator $D$ that satisfies the first order condition must be of the form:
\begin{equation}
 D(\vx + \vy) = {\l(\Lambda'(\vx) {\Lambda(0)}^\dagger\r)}^2 (D(\vy) - D(0)) + D(\vx).\label{eqn:recursion for D}
\end{equation}
\end{lemma}
\begin{proof}
 To check the the first order condition, it is sufficient to check it only for the unitary generators of $\A_{\theta}$:
 \[ [[D,U_{\vx}],U_{\vy}^o] = D U_{\vx} U_{\vy}^0 - U_{\vx} D U_{\vy}^o - U_{\vy}^o D U_{\vx} + U_{\vy}^o U_{\vx} D = 0,\]
 for all $\vx,\vy \in \Z^n$.
Using Lemma~\ref{lem:commutant property}, we write out the first order condition:
\begin{align*}
 D U_{\vx} U_{\vy}^0 e_{\vmu,j} &= a(\vx,\vy,\vmu) D(\vx + \vy + \vmu) \Lambda'(\vy){\Lambda'(0)}^{\dagger} e_{\vmu+\vx + \vy,j}\\
 U_{\vx} D U_{\vy}^0 e_{\vmu,j} &= a(\vx,\vy,\vmu) D(\vy + \vmu) \Lambda'(\vy){\Lambda'(0)}^{\dagger} e_{\vmu+\vx + \vy,j}\\
 U_{\vy}^0 D U_{\vx} e_{\vmu,j} &= a(\vx,\vy,\vmu) \Lambda'(\vy){\Lambda'(0)}^{\dagger} D(\vx + \vmu) e_{\vmu+\vx + \vy,j}\\
 U_{\vy}^0 U_{\vx} D e_{\vmu,j} &= a(\vx,\vy,\vmu) \Lambda'(\vy){\Lambda'(0)}^{\dagger} D(\vmu) e_{\vmu+\vx + \vy,j},
\end{align*}
where $a(\vx,\vy,\vmu)$ is the common factor
\[a(\vx,\vy,\vmu) = e\l(\vx\cd\am \vmu+\vmu\cd \am\vy + \vx\cd\am\vy + \half\l(\vx \cd \am\vx + \vy \cd \am\vy\r)\r).\]
 This gives the relation
\[ \l( D(\vx + \vy + \vmu) - D(\vy+ \vmu) \r) \Lambda'(\vy) {\Lambda'(0)}^\dagger = \Lambda'(\vy){\Lambda'(0)}^{\dagger} \l(D(\vx + \vmu) - D(\vmu) \r).\]
Since $D$ is self-adjoint and $\Lambda'(\vx)$ unitary, we can rewrite this as
\[ \l( D(\vx + \vy + \vmu) - D(\vy+ \vmu) \r) = {\l(\Lambda'(\vy){\Lambda'(0)}^{\dagger}\r)}^2 \l(D(\vx + \vmu) - D(\vmu) \r).\]
\end{proof}
For $\vy = \vx$ and $\vmu = 0$, the solution to the defining equation~\eqref{eqn:recursion for D} is
\begin{equation} D(\vx) = \sum_i \l(\vtau_i \cd \vx\r) A_i + {\l(\Lambda'(\vx) {\Lambda(0)}^{\dagger}\r)}^2 B + C, \label{eqn:general form of D}\end{equation}
where the $A_i$, $B$ and $C$ are bounded operators such that $C e_{\vmu,j} = \sum_k c_{jk} e_{\vmu,k}$ and similarly for $B$ and $A_i$, and $\ker\l( \sum_i \vtau_i \cd \vx A_i \r)$ contains at least the $\vx \in \Z^n$ such that ${\l(\Lambda'(\vx) {\Lambda(0)}^{\dagger}\r)}^2 \neq \Id$. This is the unique solution, since we see from equation~\eqref{eqn:recursion for D} that $D$ is fully determined after we choose suitable $D(e_i)$ for $1\leq i\leq n$.

If ${(\Lambda'(\vx) {\Lambda(0)}^\dagger)}^2 = \Id$, this gives a linear Dirac operator, familiar from commutative geometry. However, at first glance it seems that there might be other spin structures, not corresponding to commutative spin geometries. These other candidate geometries, where ${(\Lambda'(\vx) {\Lambda(0)}^\dagger)}^2 \neq \Id$, will however drop out because they are incompatible with the compact resolvent condition on $D$.

\begin{lemma}\label{lem:unique structure}
 Only if ${(\Lambda'(\vx){\Lambda(0)}^\dagger)}^2 = \Id$ can the equivariant Dirac operator $D$ have a compact resolvent. Hence $D$ is of the form:
\begin{equation}
 D e_{\vmu,j} = \l(\sum_i \l(\vtau_i \cd \vmu\r) A_i + C\r)e_{\vmu,j}. \label{eqn:special form of D}
\end{equation}
As a corollary, we have that $J$ is of the form:
\begin{equation}\label{eqn:special form of J}
 J e_{\vmu,j} = e\l(\vmu \cd \am \vmu\r) \Lambda e_{-\vmu,j},
\end{equation}
with $\Lambda$ a constant isometry such that $\Lambda^2 = \pmJJ \Id$. 
\end{lemma}

\begin{proof}
 By \cite{reed_methods_1978}*{Theorem XIII.64}, we see that an unbounded self-adjoint operator $D$ bounded away from $0$ has a compact resolvent if and only if the set \[F_b = \{ \psi \in \text{Dom}(D) : ||\psi|| \leq 1 ; ||D\psi|| \leq b\},\] is compact for all $b\in\R$. 
However, if ${(\Lambda'(\vx){\Lambda(0)}^\dagger)}^2 \neq \Id$ for some direction $\vx$, we know that $\vx \in \ker\l( \sum_i \vtau_i \cd \vx A_i \r)$ and then it follows from \eqref{eqn:general form of D} that the norm of $D e_{\lambda \vx}$ is bounded by $B + C$ for all $\lambda \in \Z$. When we take $b> ||B + C||$, $F_b$ contains at least $e_{\lambda \vx}$ for all $\lambda \in \Z \neq 0$, so $F_b$ cannot be compact.
\end{proof}

\section{Grading and Hochschild homology}\label{sec:hochschild homology}
In this section we investigate what extra conditions on the the spectral triple of the noncommutative torus come from the grading operator and Hochschild conditions, Axioms~\ref{ax:grading operator} and \ref{ax:hochschild cycle}. We find that the parameters $\vtau^i$ introduced in section~\ref{sec:dirac operator} must be linearly independent vectors spanning $\R^n$.

We start by investigating the Hochschild cycle condition, which states that there is a Hochschild cocycle $\mathbf{c} \in Z_n(\A,\A\otimes \A^o)$ whose representative on $\Hi$ is $\Gamma$ when $n$ is even, or $\Id$ if $n$ is odd. The $\Gamma$ operator is an isometry, with eigenvalues $1$ and $-1$, and so by Axiom~\ref{ax:grading operator} and the diagonal action of the algebra on the Hilbert space, we see that \[\Gamma = \begin{pmatrix} \Gamma^+ & 0\\ 0 & \Gamma^-\end{pmatrix},\] where $\Gamma^+$ and $\Gamma^-$ are unitary self-adjoint operators which have only eigenvalues $+1$ and $-1$ respectively. Using a unitary transformation, we can assume these operators to be diagonal, thus%
\[\Gamma = \begin{pmatrix}\Id_+ & 0\\ 0 & -\Id_-\end{pmatrix},\]
where $\Id_+$ and $\Id_-$ are the identity operators on the positive and negative eigenspaces of $\Gamma$.

An obvious candidate for the Hochschild cycle in $Z_n(\A,\A\otimes \A^o)$ is the straightforward generalization of the unique such cycle for the noncommutative $2$-torus \cite{paschke_spin_2006}*{Page 324}, which is \[\mathbf{c}_2 = U_{(1,0)}^* U_{(0,1)}^* \otimes U_{(0,1)} \otimes U_{(1,0)} - U_{(0,1)}^* U_{(1,0)}^* \otimes U_{(1,0)} \otimes U_{(0,1)}.\]
A candidate generator of the $n$-th Hochschild homology of the noncommutative $n$-torus is given by
\begin{equation}\label{eqn:hochschild cycle}
 \mathbf{c}_n = \sum_{\sigma \in S_n} \l(\text{sign}(\sigma) {\l(\prod_{i=1}^n U_{e_{\sigma(i)}}\r)}^* \bigotimes_{i=1}^n \l(U_{e_{\sigma(i)}}\r)\r),
\end{equation}
with $e_i$ an orthonormal basis of $\mathbb{Z}^n$. It is known \cite{bonda_elements_2001}*{Lemma 12.15} that this a Hochschild cycle. Due to \cite{wambst_hochschild_1997}*{Theorem 1.1}, the $n$-th Hochschild homology of the $n$-torus is $1$-dimensional. Together with Lemmas~\ref{lem:only nontrivial cycles} and \ref{lem:hochschild tau} below, this means \eqref{eqn:hochschild cycle}  generates the $n$-th Hochschild homology.
\begin{lemma}\label{lem:only nontrivial cycles}
 For the noncommutative $n$-torus, only nontrivial cycles can be mapped to $\Gamma$ when $n$ is even, and to $\Id$ when $n$ is odd, by the map $\pi_D$.
\end{lemma}
\begin{proof}
 Since the Hochschild cycle consists of polynomial expressions, it is enough to prove the result for individual homogeneous polynomials, since any cycle can be written as the sum of homogeneous polynomials. 
 Define \[\mathbf{c}' = U_{\vx^0} \otimes U_{\vy}^o \otimes U_{\vx^1} \otimes \cdots \otimes U_{\vx^n} \in Z_{n+1}(\A,\A\otimes \A^o),\]
with $\vx^i,\vy \in \Z^n$. As in \cite{vrilly_introduction_2006}*{Chapter 3.5}, we see that 
\begin{equation}\pi_D(b\mathbf{c}') = {(-1)}^{n} \l[ \pi^o\l(U_{\vy}\r) \pi\l(U_{\vx^0}\r) [D,\pi\l(U_{\vx^1}\r)]\ldots[D,\pi\l(U_{\vx^n}\r)],\pi\l(U_{\vx^{n+1}}\r)\r].\label{eqn:trivial hochschild cycle}\end{equation}
Since $[D,U_{\vx}] = C_{\vx} U_{\vx}$ with $C_{\vx}$ some operator depending on $\vx$, $\pi_D(b\mathbf{c}')$ is proportional to \[C_{\vx^1,\ldots, \vx^{n}} \pi\l(U_{\vx^0}\r)\pi^o\l(U_{\vy}\r) \pi\l(U_{\vx^1}\r)\ldots \pi\l(U_{\vx^n}\r)\pi\l(U_{\vx^{n+1}}\r).\]
When $n$ is even, $\Gamma$ maps $e_{\vmu,j}$ to $\pm e_{\vmu,j}$, the total sum $\vy + \sum_{i=0}^{n+1}\vx^i$ must be $0$. If $\At$ and $\At^o$ have a trivial intersection, $U_{\vy}^o$ must have total degree $0$. The total $\sum_{i=0}^{n+1}\vx^i$ is $0$, and since $U_{\vx}$ and $U_{-\vx}$ commute, the commutator~\eqref{eqn:trivial hochschild cycle} vanishes. When $n$ is odd, the argument goes the same.

If $\At$ and $\At^o$ have nontrivial intersection, there are $\vy$ for which $U_{\vx} U_{\vy} = U_{\vy} U_{\vx}$ for all $\vx \in \Z^n$. In that case, by the same arguments as for the trivial intersection case above, the $U_{\vy}$ must lie entirely within the intersection of $\At$ and $\At^o$, and since $\vx^{n+1} = -\vy - \sum_{i=0}^{n}\vx^i$, the commutator~\eqref{eqn:trivial hochschild cycle} vanishes:
\begin{align*}
 \l[U_{\vy}^o U_{\vx^1}\cdots U_{\vx^n}, U_{\vx^{n+1}}\r] &=\\
 U_{\vy}^o \l[U_{\vx^1}\cdots U_{\vx^n},U_{\vx^{n+1}}\r] &= \\
 U_{\vy}^o \l( 1 - e(\vx^{n+1} \cd \theta \sum_i^n \vx^i)\r) &=\\
 U_{\vy}^o \l( 1- e( \vx^{n+1} \cd \theta (-\vy - \vx^{n+1}))\r) &= 0.
 \end{align*}
Since $\vx^{n+1} \cd \theta \vx^{n+1} =0$ and $U_{\vy}$ commutes with $U_{\vx^{n+1}}$.
\end{proof}

\begin{lemma}\label{lem:hochschild tau}
The Dirac operator given in Lemma~\ref{lem:unique structure} satisfies the Hochschild condition only if the vectors $\vtau^i$ are linearly independent. The Dirac operator is
\begin{equation}D = \sum_i \l(\vtau^i \cd \vdelta\r) A_i\label{eqn:final dirac operator} + C,\end{equation}
where the $A_i$ are bounded operators such that 
\begin{equation}\sum_{\sigma\in S_n} \textup{sign}(\sigma) \prod_{i} A_{\sigma(i)} = \det(\vtau^1 \vtau^2\ldots \vtau^n) \Gamma,\label{eqn:gamma matrices relation}\end{equation}
when $n$ is even, and $\det(\vtau^1 \vtau^2\ldots \vtau^n) \Id$ when $n$ is odd.
\end{lemma}
 \begin{proof}
 In order to deduce the representative of the Hochschild cycle~\eqref{eqn:hochschild cycle} on the Hilbert space $\Hi$, we calculate 
\begin{equation} \begin{split}
 [D, U_{\vx} ]e_{\mu,j} &= D e(\vx \cdot \am \vmu + \half\vx \cdot \am \vx) e_{\vmu+\vx,j} - U_{\vx} \l(\sum_{k} \vtau_k \cd \vmu A_k e_{\vmu,j}\r)\\
 &= e(\vx \cdot \am \vmu + \half\vx \cdot \am \vx) \l(\sum_k \vtau_k \cdot \vx A_k e_{\vmu+\vx,j} \r).
\end{split}
\end{equation}
We suppress the $e(\vx \cdot \am \vmu + \half\vx \cdot \am \vx)$ factors, since these are canceled from the left by the $U_{e_i}^*$ factors, and expand~\eqref{eqn:hochschild cycle}:
\begin{align*}
  \pi_D(\mathbf{c}) &= \pi_D \l(\sum_{\sigma \in S_n} \l(\text{sign}(\sigma) {\l(\prod_{i=1}^n U_{e_{\sigma(i)}}\r)}^* \bigotimes_{i=1}^n \l( U_{e_{\sigma(i)}}\r)\r)\r)\\
  &= \sum_{\sigma \in S_n} \l(\text{sign}(\sigma){\l(\prod_{i=1}^n U_{e_{\sigma(i)}}\r)}^* \prod_{i=1}^n [D,U_{e_{\sigma(i)}}]\r)\\
&= \sum_{\sigma \in S_n} \l(\text{sign}(\sigma) \sum_k \prod_{i=1}^n  \vtau^{k} \cd e_{\sigma(i)} A_k \r).
\end{align*}
This expression should be some constant times $\Gamma$ or $\Id$, depending on the dimension. 
Due to \cite{connes_spectral_2008}*{Proposition 4.2}, we can write this as
\[ \pi_D (\mathbf{c}) = \det(\vtau_1 \vtau_2 \ldots \vtau_n) \sum_{\sigma\in S_n} \text{sign}(\sigma) \prod_{i} A_{\sigma(i)},\]
where we view $(\vtau_1 \vtau_2 \ldots \vtau_n)$ an $n\times n$ matrix with $\vtau_i$ as the columns, from which it is immediately clear that the vectors must be linearly independent.
\end{proof}

\section{Dimension, finiteness and regularity}\label{sec:dimension,finiteness,regularity}
Here we establish, using the conditions of dimension, regularity and finiteness (Axioms~\ref{ax:dimension}, \ref{ax:regularity} and \ref{ax:finiteness}), that the real spectral triple must be an isospectral deformation of a spin structure on a noncommutative torus. The result follows from Connes' spin manifold theorem. In the course of proving this, we find a proof for an elementary fact about Hermitian matrices generating a Clifford algebra, for which we do not know an elementary proof.

\begin{lemma}\label{lem:lattice approximation}
 If the $\vtau_i$ span $\R^n$, for arbitrary $a_i\in \R$ and all $\epsilon > 0$, there is a $t\in\R$ and a set of $\vmu_j \in \tilde{\Z}^n$, with $\tilde{\Z}^n$ a shifted lattice as in section~\ref{sec:reality operator} such that
\[ \sum_i \l| t a_i - \vtau_i \cd \sum_j \vmu_j \r| < \epsilon.\]
\end{lemma}
\begin{proof}
Every vector $\vp \in \Z^n$ can be written as a sum of at most $2$ vectors $\vmu_j \in \tilde{\Z}^n$.
By Dirichlet's theorem of simultaneous Diophantine approximation \cite{schmidt_diophantine_1991}*{Theorem II.1B}, for all $N>1$ and $(a_i')\in \R^n$ we can find integers $q, p_1, \ldots p_n$ with $q<N$, such that 
\[ | q a_i' - p_i | < N^{-1/n},\]
for $1\leq i\leq n$, and $a_i \in \R$. 

Since the $\vtau_i$ span $\R^n$, there is a transformation $\mathbf{R}\in GL(\R,n)$ depending only on the $\vtau_i$ such that ${(\mathbf{R} \vp)}_i  = \vtau_i \cd \vp$. Set $a_i' = {(\mathbf{R}^{-1} \va)}_i$. Call the eigenvalue of $\mathbf{R}$ with the maximal absolute value $\tau_{\max}$, then clearly we have
\[ |q a_i -\vtau_i \cd \sum_j \vmu_j | = | q {(\mathbf{R} \va')}_i - {(\mathbf{R} \vp)}_i | \leq |\tau_{\max}| | q a_i' - p_i| < |\tau_{\max}| N^{-1/n}.\]
Choosing $N$ such that $n |\tau_{\max}| N^{-1/n} < \epsilon$, we get the asked result.
\end{proof}

\begin{lemma}\label{lem:invertible matrices}
 In order for $D$ to be satisfy the compact resolvent condition of Axiom~\ref{ax:D compact resolvent}, and the dimension condition of Axiom~\ref{ax:dimension}, \[ \sum_i x_i A_i,\] needs to be invertible for all $x_i \in \R$ except when all $x_i$ vanish. In particular, all $A_i$ should be invertible operators.
\end{lemma}
\begin{proof}
If $\sum_i x_i A_i$ is not invertible at a point where $x_i = \vtau_i\cd \vmu$ for all $i$, then clearly this is also the case for $\vmu' = \lambda \vmu$ with $\lambda\in \Z$. Thus the kernel cannot be finite dimensional in this case. 

If $\sum_i x_i A_i$ is not invertible for some non-trivial $x_i$, but $x_i \neq \vtau_i \cd \vmu$ for all $\vmu \in\Z^n$ we have the following. 

Since the $\vtau_i$ span $\R^n$ by Lemma~\ref{lem:hochschild tau}, by Lemma~\ref{lem:lattice approximation} we have for every $\epsilon > 0$ a set of vectors $\vmu_j \in \Z^n$ such that $\sum_{i} |\sum_j \vtau_i \cd \vmu_j - x_i| < \epsilon$. Take an element $\vy$ in the kernel of $\sum_i x_i A_i$. Then $||(\sum_{i,j} \vtau_i \cd \vmu_j A_i ) \cd \vy|| < \epsilon ||\vy||$. Hence there is at least one eigenvalue of $\sum_{i,j} \vtau_i \cd \vmu_j A_i$ smaller or equal to $\epsilon$. But this means that the spectrum of $|D|^{-1}$ is unbounded, hence $D^{-1}$ cannot be compact.
\end{proof}

We have so far assumed nothing about the size of the Hilbert space $\Hi$ compared to the basic irreducible representation of the algebra, $\Hi_0$ as defined in equation~\eqref{eqn:irred hilbert space}. By construction, $\Hi$ is a left $\At$-module. According to Axiom~\ref{ax:finiteness}, a certain submodule $\Hi^{\infty}$ of $\Hi$ should be a finitely generated projective left $\At$-module. This has the following consequence:
\begin{lemma}\label{lem:finitely generated module}
 The spectral triple $(\A,\Hi,D)$ only satisfies the finiteness condition of Axiom~\ref{ax:finiteness} if the Hilbert space $\Hi$ is a finite direct sum of copies of $\Hi_0$. If we assume the algebra $\A$ to be closed under the collection of seminorms $||\delta^k (\cd)||$ defined in Axiom \ref{ax:regularity}, it is given by 
 \begin{equation} 
 \At = \l\{ \sum_{\vx\in \Z^n} a(\vx) U_{\vx} \quad | \ a(\vx) \in \mathcal{S}(\Z^n)\r\},\label{eqn:smooth algebra}
 \end{equation}
where $\mathcal{S}(\Z^n)$ is the set of Schwartz functions over $\Z^n$:
\[ \mathcal{S}(\Z^n) = \l\{ a: \Z^n \rightarrow \C \quad | \quad \sup_{\vx\in\Z^n}{(1+||\vx||^2)}^k |a(\vx)|^2 < \infty \text{ for all $k\in\N$}\r\}.\]
\end{lemma}
\begin{proof}
By the arguments in \cite{MR1303779}*{Lemma III.6.$\alpha$.2}, we know that the intersection of the domain of the $\delta^k$ is stable under the holomorphic functional calculus.
The collection of seminorms $||\delta^k(\cd)||$ is equivalent to the collection of seminorms $q_k= \sup_{\vx\in\tilde{\Z}^n}{(1+|\vx|^2)}^k |a(\vx)|$ because of the following argument.

We can write $||\delta^k(a)||^2 = \sum_i \sum_{\vx\in\Z^n} {||A_i||}^{2k}  {||\vtau_i\cd\vx||}^{2k} {|a(x)|}^2$. Because of Lemmas \ref{lem:hochschild tau} and \ref{lem:invertible matrices}, we know that $\sum_i {||A_i||}^{2k} {||\vtau_i\cd\vx||}^{2k}$ can be bounded between $c_{1,k} {||\vx||}^{2k}$ and $c_{2,k} {||\vx||}^{2k}$ for constants $c_{1,k}, c_{2,k}>0$. Thus the family of seminorms $||\delta^k(a)||$ is equivalent to the family of seminorms $q_k'(a) =\sum_{\vx\in\Z^n} {||\vx||}^{2k} {|a(\vx)|}^2$. Now if an element $a$ of the $C^*$-algebra $A$ has a finite norm in all $q_k'$, then clearly  $\sup_{\vx\in\Z^n}{(1+||\vx||^2)}^k |a(\vx)|^2 < \infty$ for all $k\in\N$. Conversely, if $a\in\At$, we can write
\begin{align*}
 q_k'(a) &= \sum_{\vx\in\Z^n} {||\vx||}^{2k} {|a(\vx)|}^2\\
&= \sum_{\vx\in\Z^n} {(1+||\vx||^2)}^{p} {(1+||\vx||^2)}^{-p} {||\vx||}^{2l} {|a(\vx)|}^2
\end{align*}
and this last sum converges if $p$ is big enough, since ${(1+{||\vx||}^2)}^{p} {|a(\vx)|}^2$ is by assumption less than some finite constant $c_p$ and $\sum_{\vx\in\Z^n} c_p {||\vx||}^{2k}{(1+||\vx||)}^{-p}$ converges when $p> n/2 + k$. 

By Axiom~\ref{ax:finiteness}, we have that $\Hi^{\infty}$ is a finitely generated projective left $\At$ module. 
We already knew that $\Hi$ must be a direct sum of copies of $\Hi_0$ due to the discussion following \eqref{eqn:direct sum}. The finiteness condition then ensures that the sum must be finite. All conditions stated in Axioms~\ref{ax:regularity} and \ref{ax:finiteness} are then easily seen to be fulfilled.
\end{proof}

\begin{remark}
 The choice of smooth structure is not unique even in the commutative case, given just the algebra $C^*$-algebra $A$ of the noncommutative torus. See for example \cite{MR0258044}. However, given the equivariance condition of the Dirac operator, and the assumption that the smooth algebra consists of all elements such that $||\delta^k(\cd)||<\infty$ for all $k\in\N$, the algebra is uniquely determined by the above argument.
\end{remark}

The number of generators of $\Hi$ in terms of $\Hi_0$ is still undetermined, but a lower bound is given by \cite{MR0179183}*{Theorem 1}. This theorem states that for complex Hermitian $k\times k$ matrices, with $k = (2 a + 1)2^b$, there exists at most $2b+1$ matrices satisfying the non-invertibility property of Lemma~\ref{lem:invertible matrices}. So in order to have at least $n$ such matrices, $k$ should be at least $2^{\lfloor n/2\rfloor}$. Hermitian matrices generating an irreducible representation of a Clifford algebra $Cl_{n,0}$ are an example of a set of matrices attaining this lower bound.

\begin{remark}It does not follow from Lemma~\ref{lem:invertible matrices} that the algebra generated by the matrices $A_i$ is a Clifford algebra. What remains to be shown is that ${(\sum_i x_i A_i)}^2$ lies in the center of the bounded operators on $\Hi$ for all $\vx \in \R^n$. Only the condition of Lemma~\ref{lem:invertible matrices} is not enough to show this.

Consider a set of self-adjoint matrices $\{ B_i\}$ generating a Clifford algebra. These satisfy the invertibility condition of Lemma~\ref{lem:invertible matrices}, and the Hochschild condition~\eqref{eqn:gamma matrices relation}. Since they are self-adjoint, we can diagonalize $B_1$ as $B_1 = U \Lambda U^\dagger$ with $\Lambda$ a diagonal matrix with real elements. If we rescale the elements of $\Lambda$ each by a different nonzero amount to $\Lambda'$, the set $\{U \Lambda' U^\dagger\} \cup {\{ B_i\}}_{i\geq 2}$ still has the invertibility property, but not necessarily the Hochschild property, as a calculation for any case $n\geq 4$ will show. If $n=2,3$, the invertibility property does imply the Hochschild property however.
\end{remark}

However, a weaker form of Connes' reconstruction theorem \cite{connes_spectral_2008}*{Theorem 11.5}, implies the following result:
\begin{lemma}\label{lem:clifford action}
 In order for the candidate structure $(\At, \Hi, D, J)$ to satisfy both the Hochschild condition of Axiom~\ref{ax:hochschild cycle} and the dimension condition of Axiom~\ref{ax:dimension}, the matrices $A_i$ of Lemma~\ref{lem:unique structure} must generate a Clifford algebra.
\end{lemma}

\begin{proof}
 If we look at our conditions on $\At$, $\Hi$, $D$ and $J$, we see that none of them depend on the antisymmetric matrix $\theta$. Also, the action of the Dirac operator on the Hilbert space is independent of $\theta$. This means that we can just set $\theta = 0$, where we have the real spectral triple $(\A_0, \Hi, D,J)$ of smooth functions on the $n$-torus. Due to the results of Connes' spin manifold theorem (see for example \cite{bonda_elements_2001}*{Lemma 11.6}, \cite{connes_spectral_2008}*{Remark 5.12}), this implies that the $A_i$ generate a Clifford algebra.
\end{proof}
Because the size of the maximal set of matrices which satisfy the invertibility condition of Lemma~\ref{lem:invertible matrices} is odd, due to \cite{MR0179183}*{Theorem 1}, we have as a corollary:
{\renewcommand\thecorollary{\ref{cor:matrix theorem}}
\begin{corollary}
A set of $2^b\times 2^b$ Hermitian matrices ${\{A_i\}}_{i=1}^n$, where $n=2b +1$, such that the equation
\[\det \l(\sum_i x_i A_i\r) = 0,\]
only has the zero solution ${(x_i =0)}_{i=1}^n$ in $\R^n$ generate a Clifford algebra if and only if
\[ \sum_{\sigma \in S_n} \text{sign}(\sigma) \prod_{i}^n A_{\sigma(i)} = \lambda \Id_k,\]
for some nonzero $\lambda \in \R$.
\end{corollary}
\addtocounter{corollary}{-1}}
If we assume the spectral triple to be irreducible, we need that the matrices generate an irreducible representation of the  Clifford algebra. This restricts the size of the matrices to be exactly $2^{\lfloor n/2 \rfloor}$.

From the results in sections~\ref{sec:hilbert space and algebra}, \ref{sec:reality operator} and \ref{sec:dirac operator} it now follows that the remaining conditions, the Poincar\'e duality of Axiom~\ref{ax:poincare} and the Hermitian pairing of Axiom~\ref{ax:finiteness} are also satisfied, since they are satisfied by isospectral deformations.

This completes the proof of Theorem~\ref{thm:main theorem}:
{\renewcommand\thetheorem{\ref{thm:main theorem}}
\begin{theorem}
All irreducible real spectral triples with an equivariant $n$-torus actions are isospectral deformations of spin structures on an $n$-torus.
\end{theorem}
\addtocounter{theorem}{-1}}

\section{Description of the real spectral triples}\label{sec:uniqueness and freedom}
Finally, we show that given a Dirac operator $D$ that satisfies all conditions so far, the reality operator $J$ is uniquely determined, and list all the ingredients that constitute all real spectral triples of the noncommutative $n$-torus. Also, we show in some low dimensional cases what freedom there still exist in the definition of the Dirac operator.

Recall that the Clifford algebra $Cl_{n,0}$ is the algebra over $\R$ generated by $\R^n$, $1$ and a positive definite quadratic form $q$, subject to the relation $v\cdot v = - q(v)\Id$. The Clifford \emph{group} is defined as the group generated by the image of an orthonormal basis of $\R^n$ together with $-1$. 
\begin{lemma}\label{lem:uniqueness of J}
 If $D$ is given by $\sum_{j} \l(\vtau^j \cdot \vdelta\r) A_j$ with $A_j$ representatives of the Clifford group of $Cl_{n,0}$ there is a unique $J$ operator for each $n$, up to multiplication with a complex number of norm $1$.
\end{lemma}
\begin{proof}
If $d=1,2,3$ or $4$ this can be done for example by calculations, see Remark~\ref{rem:low-dimensional J operators} below. We proceed by induction. First we prove existence.
Recall that there are isomorphisms $Cl_{n+2,0} \simeq Cl_{0,n} \otimes Cl_{2,0}$ and $Cl_{0,n+2} \simeq Cl_{n,0} \otimes Cl_{0,2}$ \cite{lawson_spin_1989}*{Theorem I.4.1}.
Let $d>4$, and assume it has been proven for $n-4$. The operator $J_{n} = J_{n-4} \otimes J_4$, acting on $Cl_{n,0} \cong Cl_{n-4,0} \otimes Cl_{4,0}$ has precisely the right commutation relations, except for $n \equiv 1\mod 4$, as can easily be calculated by looking at Table~\ref{tab:signs of triple}, and taking into account the periodicity mod $4$ of the table, except for the first row, where we use $J_4^2 = -1$. In case $n \equiv 1\mod 4$, we can achieve the same by setting $J_{n} = J_{n-4} \otimes \Gamma_4 J_4$. 

Now we prove the uniqueness. Write $\gamma^i_n$ for the representation of the $i$-th basis vector of $\R^n$ in $Cl_{n,0}$. An explicit isomorphism $Cl_{n,0} \equiv Cl_{n-4,0} \otimes Cl_{4,0}$ can be chosen, for example 
\[\gamma^i_n = \Id_{n-4} \otimes \gamma^i_4 \text{ for $i\leq4$, } \gamma^i_n = \gamma^{i-4}_{n-4} \otimes \gamma^1_4 \gamma^2_4 \gamma^3_4 \gamma^4_4\text{ for $i>4$.}\] The operators $\gamma^1_n \gamma^2_n$ and $\gamma^3_n \gamma^4_n$ commute with $\gamma^i_n$ for $i>4$ and anticommute with $\gamma^1_n,\gamma^2_n$ and $\gamma^3_n,\gamma^4_n$ respectively. They square to $-1$, so the operator
\[P^+ := \frac{1}{4} \l(1 + i \gamma^1_n \gamma^2_n\r) \l(1 + i \gamma^3_n \gamma^4_n\r),\]
is a projection, which commutes with $\gamma^i_n$ for $i>4$. The projection $P^+$ does not commute with $J_n$, but does commute with $\gamma_2 \gamma_3 J_n$. Also, the projection $P^+$ projects onto a subspace of dimension $1/4$ times the dimension of the irreducible representation of $Cl_{n,0}$, and the operator $\gamma_2 \gamma_3 J_n$ has the same commutation relations with $\gamma^i_n$ for $i>4$ as $J_n$, and since ${(\gamma_2 \gamma_3 J_n)}^2 = -J_n^2$ it has the right signs for an $(n-4)$-dimensional $J$ operator. This means that $P^+$ projects onto a Hilbert space belonging to an $(n-4)$-dimensional spectral triple, where we have a unique $J$ operator by the induction hypothesis. On the complement of the $P^+$ eigenspace, we have the unique $J_4$ operator with the right commutation relations with $\gamma_i$, $i\leq 4$.
\end{proof}

Stated more elaborately, we have the following result:
\begin{theorem}\label{thm:elaborate main theorem}
The following give all $2^n$ irreducible real spectral triples on the smooth noncommutative $n$-torus $\At$:%
 \begin{itemize}
 \begin{subequations}
  \item A Hilbert space $\Hi$ constructed as follows:
  \begin{equation}
   \Hi = \bigoplus_{i}^{2^{\lfloor n/2 \rfloor}} \Hi^i \quad \Hi^i = \bigoplus_{\vm\in \Z^n +\veps} \C,
  \end{equation}
  with $\veps = (\epsilon_1,\ldots,\epsilon_n)\in \R^n$, and $\epsilon_i \in \{0,\half\}$.
  \item An involutive algebra $\A$ with unitary generators $U_{\vx}$ with $\vx \in \Z^n$
  \begin{equation}
   \At := \{ A = \sum_{\vx} a(\vx) U_{\vx} : a \in \mathcal{S}(\Z^n)\},
  \end{equation}
  with $U_{\vx}$ acting on a basis vector $e_{\vmu,i} \in \Hi^i$ by 
  \begin{equation}
   U_{\vx} e_{\vmu,i} = e\l( \half \vx \cd \am\vx + \vx\cd \am \vmu\r) e_{\vmu+\vx,i}\label{eqn:algebra action},
  \end{equation}
   for any matrix $\am$ such that $\am - \am^t = \theta$.
   \item An unbounded, densely defined, self-adjoint first order operator $D$
   \begin{equation}
    D = \sum_{j=1}^n (\vtau^j\cdot \vdelta) A_j + C\label{eqn:dirac operator},
   \end{equation}
    acting on $\Hi$ with $C$ a bounded self-adjoint operator commuting with the algebra satisfying $J C J^{-1} = \pmJD C$, and $\Gamma C = - C\Gamma$ if $n$ is even, for $\vtau^j$ $n$ linearly independent vectors in $\R^n$, $n$ matrices $A_j$ of size $2^{\lfloor n/2\rfloor} \times 2^{\lfloor n/2\rfloor}$  generating an irreducible representation of the Clifford group $Cl_{n,0}$, and $\vdelta$ the derivations $\delta_i e_{\vmu} = \mu_i e_{\vmu}$.
    \item If $n$ is even, the grading operator $\Gamma$ is given by $\sum_{\sigma \in S_n} \text{sign}(\sigma) \prod_i A_{\sigma(i)}$, with  $A_j$ the matrices given above. 
    \item A unique (up to multiplication with a complex number of modulus \(1\)) antilinear isometry $J$ that acts as \begin{equation}
	J e_{\vmu,j} =  e(\vmu \cd \am \vmu) \Lambda e_{-\vmu,j}\label{eqn:reality operator}.
    \end{equation}
  with $\Lambda$ a bounded linear operator such that $\Lambda \Lambda^{\dagger} = \Id$ and $D\Lambda = -\pmJD \Lambda D^*$.
 \end{subequations}
 \end{itemize}
\end{theorem}

In lower dimensional cases we can explicitly calculate what form $C$ in \eqref{eqn:dirac operator} can take. In \cite{paschke_spin_2006}*{Lemma 2.3, Theorem 2.5} it is proven that $C=0$ if $n=2$. For $n=3$ we have the following:

\begin{proposition}\label{prop:uniqueness in low dimensions}
 If $n=3$ the constant matrix $C$ must have the form $q\Id$ where $q\in \R$ arbitrary.
 If $n=4$, the matrix $C$ must have the form \[\begin{pmatrix} 0 & 0 & a & b\\ 0 & 0 & - \bar{b} & \bar{a}\\ \bar{a} & -b & 0 & 0\\ \bar{b} & a & 0 & 0\end{pmatrix},\]
where $a,b\in \C$.
\end{proposition}

\begin{proof}
If $n=3$ it follows from Theorem~\ref{thm:elaborate main theorem} that
\[ J = e(\vmu \cd \am \vmu) \sum_{k\in I} a_{jk} e_{-\vmu,k},\]
with $I= \{1,2\}$. We choose a particular form of $D$, given by the particular representation of the Clifford group $Cl_{3,0}$ known as the Pauli matrices. We see that
\[J = e(\vmu \cd \am \vmu) \begin{pmatrix}
                           0 & -1\\
			   1 & 0
                          \end{pmatrix},\]
just as in the $n=2$ case in \cite{paschke_spin_2006}*{Theorem 2.5}. Using the appropriate values for $\pmJJ$ and $\pmJD$ found in Table~\ref{tab:signs of triple}, we see that the defining equation is $J C J = C$ and by a calculation this shows that $C = q \Id$ where $q \in \R$ arbitrary.
Similarly for $n=4$, we can just check what the conditions are for $C$ to satisfy the equations $\Gamma D = - D \Gamma$ and $J C = C J$, and this gives the possibilities given in the proposition. It is trivial to calculate similar conditions for higher dimensions. Due to the increase in the size of matrices, and relaxation of the commutation relation with $\Gamma$ when going from $n=2k$ to $n=2k+1$, the number of parameters will increase when $n$ grows.
\end{proof}

\begin{remark}
If we choose a representation for $Cl_{1,0}$, $Cl_{2,0}$ and $Cl_{0,2}$ all Clifford algebras $Cl_{n,0}$ can be constructed by the basic isomorphisms \cite{lawson_spin_1989}*{Theorem I.4.1}:
\begin{align*}
 Cl_{n,0} \otimes Cl_{0,2} &\cong Cl_{0,n+2},\\
 Cl_{0,n} \otimes Cl_{2,0} &\cong Cl_{n+2,0}.\\
\end{align*}
We choose a representation:
\begin{align*}
Cl_{1,0} &= 1,\\
Cl_{2,0} &= \begin{pmatrix} 0 & i\\ -i & 0\end{pmatrix}, \begin{pmatrix} 1 & 0\\ 0 & -1\end{pmatrix},\\
Cl_{0,2} &= \begin{pmatrix} 0 & -1\\ 1 & 0\end{pmatrix}, \begin{pmatrix} i & 0\\ 0 & -i\end{pmatrix},
\end{align*}
the (unique up to multiplication with a complex number of norm \(1\)) matrix part of the $J$ operator can easily be calculated:%
\begin{subequations}
\begin{align}
 J_2 &= J_3 = \begin{pmatrix} 0 & 1\\ -1 & 0\end{pmatrix}\\
 J_4 &= \begin{pmatrix} 0 & 1 & 0 & 0\\ -1 & 0 & 0 & 0\\ 0 & 0 & 0 & -1\\ 0 & 0 & 1 & 0\end{pmatrix}\\
 J_5 & = \begin{pmatrix} 0 & 0 & -1 &0\\ 0 & 0 & 0 & 1\\ 1 & 0 & 0 & 0\\ 0 & -1 & 0 & 0\end{pmatrix}.
\end{align}
\end{subequations}\label{rem:low-dimensional J operators}
In this representation, we also see that for all $n$, the matrix component of $J$ has precisely one nonzero element in every column or row.
\end{remark}

\section{Unitary equivalences}\label{sec:unitary equivalences}
After establishing that all real spectral triples on the noncommutative $n$-torus are isospectral deformations of spin structures on the commutative $n$-torus, we show that there is a big difference between their groups of symmetries. In the commutative case, a diffeomorphism acting from the torus can transform a spin structure into another according to the action of the diffeomorphism group as calculated in \cite{MR860317}. Here we show that in the noncommutative case when $n=2$ most spin structures, except the isospectral deformation of the trivial spin structure, are equivalent. When $n>2$ the result is less conclusive, due to insufficient knowledge of the full automorphism group of the $C^*$-algebra in that case.

Recall the definition of a unitary equivalence between two spectral triples:
\begin{definition}[Unitary equivalence]
A unitary equivalence between two spectral triples $(\A,H,D,J,(\Gamma))$ and $(\A,H,D',J',(\Gamma'))$ is given by a unitary operator $W$ acting on the Hilbert space $H$ such that%
\begin{subequations}\label{eqn:unitary transformation}
\begin{align}
  W \pi(a) W^{-1} &= \pi(\sigma(a)) \quad\forall a \in \A\label{eqn:action of W on algebra},\\
  W D W^{-1} &= D'\label{eqn:action of W on D},\\
  W J W^{-1} &= J'\label{eqn:action of W on J},\\
W \Gamma W^{-1} &= \Gamma'\label{eqn:action of W on Gamma},
\end{align}
\end{subequations}
where $\sigma$ is a $*$-automorphism of the $C^*$-algebra $A$ such that the algebra $\A$ is mapped into itself.
\end{definition}
We first recall what is known about the automorphisms of the $C^*$-algebra $\Atc$.
The main tool for understanding the automorphism group of a general noncommutative $n$-torus is \cite{MR1488068}*{Theorem I}, which tells us that for $\theta$ in a set which has full measure in the space of all antisymmetric matrices, the algebra $\Atc$ is an inductive limit of direct sums of circle algebras. This results allows one to generalize a lot of results on noncommutative two-tori to higher dimensional tori.

For a noncommutative $n$-torus which is an inductive limit of direct sums of circle algebras, the automorphism group fits in the following exact sequence \cite{elliott_automorphism_1993}*{Theorem 2.1}:
\[1 \rightarrow \overline{\text{Inn}}(\Atc) \rightarrow \text{Aut}(\Atc) \rightarrow \text{Aut}(K(\Atc)) \rightarrow 1,\]
where an automorphism of $K_0 (\Atc) \oplus K_1(\Atc)$ should preserve the order unit $[1_{\Atc}]$ and the order structure ${\l(K_0 (\Atc) \oplus K_1(\Atc)\r)}^+$.

The left side of the exact sequence can be further specified by \cite{elliott_automorphism_1993}*{Corollary 4.6}, which states that for algebras which are inductive limits of direct sums of circle algebras $\overline{\text{Inn}}(\Atc) = \overline{\text{Inn}}_0 (\Atc)$. This means that an inner automorphism of a noncommutative torus can only give a unitary equivalence of two spin structures if the identity automorphism gives a unitary equivalence between the two spin structures.

By \cite{rieffel_projective_1988}*{Theorem 6.1} the order structure ${(K_0(\At))}^+$ consists precisely of those elements for which the normalized trace is positive, and by \cite{elliott_k-theory_1984}*{Theorem 3.1}, the image of this trace on $K_0$ is equal to the range of the exterior exponential of $\theta$:
\[ {\exp}\bigwedge \theta = 1 \oplus \theta \oplus \half (\theta \wedge \theta)\oplus \ldots : \bigwedge^{\text{even}} \Z^n \rightarrow \R.\]
For noncommutative $2$-tori with $\theta$ irrational, this means that all automorphisms of $K(\Att)$ must be the identity on $K_0(\Att) = \Z + \theta\Z$. The $K_1$ groups for noncommutative $n$-tori are also known to be isomorphic to $\Z^{2^{n-1}}$. For $n=2$ a partial lifting of $\text{Aut}(\Z^2) = \text{GL}(2,\Z)$ is known \cite{MR734854}, and given by the action of $\text{SL}(2,\Z)$ on the lattice $\Z^2$ of unitary generators $U_{\vx}$. In \cite{MR859436} it is proven that in fact the whole automorphism group of the algebra $A(\mathbb{T}_\theta^2)$ for irrational $\theta$ with certain extra Diophantine conditions is given by a semidirect product of this action, the canonical torus action: $U_{e_i} \mapsto \lambda_i U_{e_i}$ and the projectivized group of unitaries of $A(\mathbb{T}_\theta^2)$ in the connected component of the identity:
\[ {PU(A(\mathbb{T}_\theta^2))}^0 \rtimes (\mathbb{T}^2 \rtimes \text{SL}(2,\Z)).\]

For $n>2$ the situation is less clear, since the action on the $K_0$ group need not be trivial anymore.
When $n>2$, if $\theta$ is a matrix such that all $\theta_{ij}$ are independent over $\Z$, one can easily calculate that an outer automorphism cannot simply map basic unitaries to other basic unitaries. An automorphism $\sigma$ of this form must satisfy $e(\sigma(\vx)\cd\theta \sigma(\vy)) = e(\vx\cd\theta\vy)$ for all $\vx,\vy\in\Z^n$. If the $\theta_{ij}$ are independent over $\Z$, we see that $\sigma(\vx) = \alpha_{11} \vx + \alpha_{12} \vy$ and $\sigma(\vy) = \alpha_{21} \vx + \alpha_{22} \vy$ with $\begin{pmatrix} \alpha_{11} & \alpha_{12}\\ \alpha_{21} & \alpha_{22} \end{pmatrix} \in \text{SL}(2,\Z)$. When $n>2$ the only solution for all $\vx$ and $\vy$ is $\sigma(\vx) = \pm \vx$ with sign the same for all $\vx$.

With this knowledge of the automorphism group, we can proceed to our main theorem of this section.
{\renewcommand\thetheorem{\ref{thm:unitary inequivalence}}
\begin{theorem} Except for a set of $\theta$ of measure $0$, the different spin structures of the smooth noncommutative $n$-torus $\At$ cannot be unitarily equivalent by an inner automorphism of the algebra.
\end{theorem}
\addtocounter{theorem}{-1}
}
In the case $n=2$ the theorem was proven in \cite{paschke_spin_2006}*{Theorem 2.5}.

\begin{proof}
While for the proof of Theorem~\ref{thm:unitary inequivalence} it is only necessary to consider inner automorphisms, we get Corollary~\ref{cor:outer automorphism} if we also consider automorphisms induced by an action of $SL(2,Z)$ on the algebra.

By \cite{elliott_automorphism_1993}*{Corollary 4.6} and \cite{MR1488068}*{Theorem I}, we have that for almost all noncommutative tori $\overline{\text{Inn}}(\Atc) = \overline{\text{Inn}}_0(\Atc)$. This means that if an inner automorphism changes the spin structure, than so does the identity automorphism. We will assume in the following that the components of $\theta$ in the upper right corner are independent over $\Z$.
The set of $\theta$ all of whose components in the upper right corner are independent over $\Z$ is of full measure, and so is the set of $\theta$ for which $\overline{\text{Inn}}(\At) = \overline{\text{Inn}}_0(\At)$, so their intersection also has full measure.

We label the basis of the Hilbert space for the different spin structures by the same labels $\vm$, so $\vmu = \vm+\veps$. We see that for a spin structure $\veps$ the operator $J$, written as in equation~\eqref{eqn:special form of J}, acts as
\[ J_{\veps} e_{\vm,i} = \Lambda_{ij} e\l((\vm + \veps)\cdot \am (\vm + \veps)\r) e_{-\vm+2 \veps,j}.\]

We consider a unitary transformation $W$, induced by an automorphism $\sigma\in SL(2,Z)$, and denote
\[W e_{\vo,i} = \sum_{\vk} \sum_j w_{\vk,ij} e_{\vk,j},\]
for the action on the $e_{\vo,i}$, and write $\sigma(\vx)$ for the obvious action of the automorphism $\sigma$ (either an element of $SL(2,\Z)$ when $n=2$, or the action $\vx \mapsto \pm \vx$ when $n>2$) on the vector $\vx \in \Z^n$.

We first show that the action of the unitary transformation $W$ on the Hilbert space is fully determined by the action on the different $e_{\vo,i}$. Next we show that the condition~\eqref{eqn:action of W on J} implies that the action of the unitary transformation on the basis vectors must be such that a basis vector $e_{\vo,i}$ is mapped to a linear combination of vectors $e_{\vk,l}$ with $\vk =\vepst - \sigma(\veps)$, where $\veps$ is the original spin structure and $\vepst$ the new spin structure. If $\sigma(\veps) = \pm \veps$, this means that the spin structure is unchanged, since $\vk$ must lie in $\Z^n$. For the noncommutative $2$-torus, we have $\sigma\in SL(2,\Z)$, and we see that if $\veps \in \Z^2$, the spin structure is unchanged. All spin structures on the noncommutative $2$-torus for which $\veps \notin \Z^2$ are unitarily equivalent.

Consider a unitary transformation $W$ that maps a spin structure $\veps$ to $\vepst$.
Since $e_{\vx,i} = U_{\vx} e(-\vx \cd \am \vx /2 - \vx \cd \am \veps) e_{\vo,i}$ and using \eqref{eqn:action of W on algebra} we can write
\begin{align*}
 W e_{\vm,i} &= W U_{\vm} e(-\vm \cd \am \vm /2 - \vm \cd \am \veps) e_{\vo,i}\notag\\
 &= U_{\sigma(\vm)} e(-\vm \cd \am \vm /2 - \vm \cd \am \veps) \l(\sum_{\vk,j} w_{\vk,ij} e_{\vk,j}\r)\notag\\
&= \sum_{\vk,j} e\l(\sigma(\vm)\cdot \am (\vk + \vepst) + \half \sigma(\vm) \cd \am \sigma(\vm) -\vm \cd \am \vm /2 - \vm \cd \am \veps\r) w_{\vk,ij} e_{\vk+\sigma(\vm),j}.
\end{align*}
So the action of $W$ on the Hilbert space is fully determined by the action on the $e_{\vo,i}$.
Requirement~\eqref{eqn:action of W on J} gives the following equations:
\begin{align*}
W J_{\veps} e_{\vm,i} &= \Lambda_{ij} e\l((\vm+\veps) \cdot \am (\vm+\veps)\r) W e_{-\vm + 2 \veps,j}\\
&= \sum_{\vk,l}\Lambda_{ij} e\l((\vm+\veps) \cdot \am (\vm+\veps) + \sigma(-\vm + 2 \veps) \cd \am (\vk + \vepst + \half\sigma(-\vm +2 \veps))\r)\\
&\cd e\l( - (2 \veps - \vm)\cd \am (\veps + \half (2 \veps -\vm))\r) w_{\vk,jl} e_{\vk + \sigma(-\vm + 2 \veps),l}\displaybreak[0]\\
&= \sum_{\vk,l} \Lambda_{ij} e\l(\veps\cd \am (-3 \veps + 2\vm) + \sigma(\veps)\cd \am (2 (\vepst + \vk) + \sigma(2\veps - \vm))\r)\\
&\cd e\l(\vm \cd \am (3 \veps + \half \vm) + \sigma(\vm)\cd \am(-\vepst - \vk + \sigma(\half \vm - \veps))\r)w_{\vk,jl} e_{\vk + \sigma(-\vm + 2 \veps),l}\displaybreak[1]\\
J_{\vepst} W e_{\vm,i} &= \sum_{\vk,j} e\l(\sigma(\vm)\cd \am (\vk + \vepst +\sigma(\vm)/2 ) - \vm \cd \am (\veps + \vm/2)\r) w_{\vk,ij} e_{\vk + \sigma(\vm),j} \\
&= \sum_{\vk,j} \Lambda_{jl} e\l( (\vk + \sigma(\vm)+\vepst) \cd \am (\vk + \sigma(\vm)+\vepst)\r)\\
&\cdot e\l(-\sigma(\vm)\cd \am (\vk + \vepst +\sigma(\vm)/2 ) + \vm \cd \am (\veps + \vm/2)\r) w_{\vk,ij}^* e_{-\vk - \sigma(\vm)+ 2 \vepst,j}\displaybreak[0]\\
&= \sum_{\vk,j} \Lambda_{jl} e\l( (\vepst + \vk)\cd \am (\vepst + \vk + \sigma(\vm)) + \vm\cd \am (\veps + \vm/2)\r)\\
&\cdot e\l( \half \sigma(\vm)\cd \am \sigma(\vm)\r) w_{\vk,ij}^* e_{-\vk - \sigma(\vm)+ 2 \vepst,j}.
\end{align*}
Collecting the vectors with the same indices and in the same Hilbert space $\Hi^j$ we see that for indices $\vk+\sigma(-\vm+2\veps) = -\vk'-\sigma(\vm)+2\vepst$, or $\vk' = -\vk + 2\vepst-2\sigma(\veps)$:
\begin{align*}
&e\l( \sigma(-\vm + 2 \veps) \cd \am (\vk + \vepst + \half\sigma(-\vm +2 \veps)) - (2 \veps - \vm)\cd \am (\veps + \half (2 \veps -\vm))\r) w_{\vk,jl}\\
&= \Lambda_{jl} e\l( (\vepst + \vk)\cd \am (\vepst + \vk + \sigma(\vm)) + \vm\cd \am (\veps + \vm/2) + \half \sigma(\vm)\cd \am \sigma(\vm)\r) w_{\vk,ij}^*.
\end{align*}
Since $\Lambda^{-1} = \Lambda^\dagger$ and $\sum_j \Lambda_{lj}\Lambda_{ji}^* = \pm \delta_{il}$ with $\delta_{il}$ the Kronecker delta, this implies that
\begin{align}
w_{\vk,ij} =& \pm e\l(2 \veps\cd\am(3 \veps - 2\vm) + 3 \vepst \cd\am(3\vepst -\vk +\sigma(-2\veps + \vm))\r)\notag\\
&\cdot e\l(\vk \cd\am(-3 \vepst + \vk +\sigma(2 \veps -\vm)) + \sigma(\veps)\cd\am(-4 \vepst + 4 \vk + 3 \sigma(2\veps -\vm))\r)\notag\\
&\cdot e\l(\sigma(\vm)\cd \am (\vepst + \vk +\sigma(\veps)) - 2 \vm\cd\am\veps\r)w_{-\vk+2\vepst -2\sigma(\veps),ij}^*.\label{eqn:w first iterate}
\end{align}
Applying this same formula again for $ w_{-k + 2\vepst - 2\sigma(\veps),ij}^*$, we get
\begin{align}
w_{-\vk + 2\vepst-2\sigma(\veps),ij}^* &= e\l( 
            \veps\cd\am(3\veps -2\vm) - (\vepst+\vk)\cd\am(\vepst+\vk+\sigma(\vm)) - \vm\cd\am(4 \veps +\vm)\r)\notag\\  
	    &\cd e\l(\sigma(\veps)\cd\am(6 \vepst -2\vk - \sigma(2\veps -\vm)) + \sigma(\vm)\cd\am(3\vepst - \vk - \sigma(\veps +\vm))\r)
w_{\vk,jk}.\label{eqn:w second iterate}
\end{align}
Filling in the expression of \eqref{eqn:w second iterate} in \eqref{eqn:w first iterate}, we see
\begin{align*}
w_{\vk,jk} &=e\l(2\vepst\cd\am(4\vepst - 2\vk +\sigma(\vm-3\veps)) + 2\vk\cd\am(\sigma(\veps-\vm)-2\vepst)\r)\\
&\cdot e\l( 2 \sigma(\veps)\cd\am(3\vk -5 \vepst +\sigma(4\veps -\vm)) + 2\sigma(\vm)\cd\am(\vk-\vepst+\sigma(\veps))\r)w_{\vk,jk}.
\end{align*}
Collecting all terms which contain $\vm$ we see that these add up to $2 \sigma(\vm) \cd \theta (\vk + \sigma(\veps) - \vepst)$.
 Since the equality above should hold for all $\vm$, this means that either $w_{\vk,ij} = 0$, or $\vm \cd \theta (\vk + \sigma(\veps) -\vepst) = 0$ for all $\vm$.

If all components of $\theta$ in the upper right corner are independent over $\Z$ this can only be the case if $\vk + \sigma(\veps) -\vepst = 0$, hence $\vk = \vepst - \sigma(\veps)$. Since $\vk$ must lie in $\Z^n$, we see that if $\sigma(\veps) = \pm \veps$ the spin structure cannot change.
\end{proof}

From the proof it also follows:
\begin{corollary}\label{cor:outer automorphism}
Let $\sigma\in SL(2,\Z)$. The automorphism $\sigma$ of $\Z^2$ induces an automorphism of the noncommutative $2$-torus of the form $U_{\vx} \mapsto  U_{\sigma(\vx)}$.

This automorphism induces a unitary equivalence of real spectral triples which maps a spin structure $\veps$ to $\vepst = \sigma(\veps)$. In particular, real spectral triples on the noncommutative $2$-torus which are not isospectral deformations of the trivial spin structure on the commutative $2$-torus are unitary equivalent to each other, via the following unitary map $W$:
\begin{align*}
 W e_{\vmu} &= e_{\sigma(\vmu)},\\
 W U_{\vx} W^{-1} &= U_{\sigma(\vx)},\\
 D' &= \l(\sum_j \sigma^{-1}(\vtau_j) \cdot \vdelta\r) \otimes A_j,\\
 J' &= J,
\end{align*}
composed with an additional unitary map given by Lemma~\ref{lem:algebra representation} that maps the representation $\pi^{\am'}$ with $\am' = \sigma^t \am \sigma$ to the original $\pi^{\am}$.
\end{corollary}
\begin{proof}
 The statement about the spin structures follows from the statements at the end of the proof of Theorem~\ref{thm:unitary inequivalence}. Since $\sigma\in SL(2,\Z)$ is an automorphism of the lattice $\Z^2$, it cannot map an $\veps \notin \Z^2$ to one in $\Z^2$ and vice-versa. On the other hand, set $\veps_{10} =  \begin{pmatrix}\half\\ 0\end{pmatrix}, \veps_{01} =  \begin{pmatrix}0\\ \half\end{pmatrix}, \veps_{11} = \begin{pmatrix}\half\\ \half\end{pmatrix}$, and $M = \begin{pmatrix} 1 & 0\\ -1 & 1\end{pmatrix}, N = \begin{pmatrix} 1 & -1\\ 0 & 1\end{pmatrix}$ in $SL(2,\Z)$. We see 
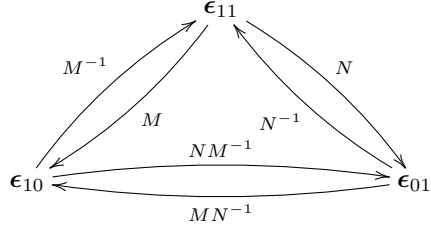
\[ \xymatrix{ & & \veps_{11} \ar@/^/[ddll]^M \ar@/^/[ddrr]^N & & \\ 
 & & & & \\
 \veps_{10} \ar@/^/[uurr]^{M^{-1}} \ar@/^/[rrrr]^{NM^{-1}} & & & & \veps_{01} \ar@/^/[llll]^{MN^{-1}} \ar@/^/[uull]^{N^{-1}}}
\]
where the matrices $M$ and $N$ act by left multiplication, so there exist $\sigma\in SL(2,\Z)$ such that the $\veps\notin \Z^2$ are mapped to one another.
 Given a unitary operator $W e_{\vmu} = e_{\sigma(\vmu)}$, we can calculate its action on the Dirac operator by using the definition \eqref{eqn:action of W on D}:
\begin{align*}
 D' e_{\vmu,k} &= W D W^{-1} e_{\vmu,k}\\
 &= W D e_{\sigma^{-1} (\vmu),k}\\
 &= W \sum_j \l(\vtau_j \cd (\sigma^{-1}(\vmu))\r) \otimes A_j e_{\sigma^{-1}(\vmu),k}\\
 &= \sum_j \l(\sigma^{-1}(\vtau_j) \cd \vmu\r) \otimes A_j e_{\vmu,k}.
\end{align*}
The action of the automorphisms $\sigma \in \text{SL}(2,\Z)$ on the Dirac operator when $\veps = 0$ was also determined for $\veps = 0$ in \cite{vrilly_introduction_2006}*{Section 7.1}. with the change in notation that our $\vtau_1$ is given there by $\begin{pmatrix} 0\\ \text{Im} \tau\end{pmatrix}$ and our $\vtau_2$ is there $\begin{pmatrix} 1\\ \text{Re} \tau\end{pmatrix}$. 
\end{proof}

Some questions left unanswered by these results are the effects of the outer automorphisms for $n>2$ and Morita equivalences of noncommutative tori, as described in \cite{rieffel_morita_1999} and \cite{elliott_morita_2007}, on the spin structure. Also, the definition of an equivariant spectral triple can be generalized to allow $[D,h]\neq 0$, but bounded. This was used to construct equivariant $\mathcal{U}_q(SU(2))$ spectral triples in \cite{MR2308305}. It would be interesting to investigate what possibilities this would open up for the noncommutative torus. 
We hope to return to these questions in the future.

\def\polhk#1{\setbox0=\hbox{#1}{\ooalign{\hidewidth
  \lower1.5ex\hbox{`}\hidewidth\crcr\unhbox0}}}

\end{document}